\documentclass[11pt,leqno]{article}
\usepackage{hyperref}
\usepackage{soul}
\usepackage{appendix}
\usepackage{pdfsync}
\usepackage{pifont}
\usepackage{rotating}
\usepackage[doc]{optional}
\usepackage{srcltx}
\usepackage{xcolor}
\definecolor{labelkey}{rgb}{0,0.08,0.45}
\definecolor{rekey}{rgb}{0,0.6,0.0}
\definecolor{Brown}{rgb}{0.45,0.0,0.05}
\usepackage{exscale,relsize}
\usepackage{amsmath}
\usepackage{amsfonts}
\usepackage{amssymb}
\usepackage{calc}
\usepackage{theorem}
\usepackage{pifont}      
\usepackage{graphicx}
\newtheorem{theorem}{Theorem}[section]
\newtheorem{lemma}[theorem]{Lemma}
\newtheorem{fact}[theorem]{Fact}
\newtheorem{corollary}[theorem]{Corollary}
\newtheorem{proposition}[theorem]{Proposition}
\newtheorem{definition}[theorem]{Definition}

\theoremstyle{plain}{\theorembodyfont{\rmfamily}
}
\theoremstyle{plain}{\theorembodyfont{\rmfamily}
}
\theoremstyle{plain}{\theorembodyfont{\rmfamily}
}
\theoremstyle{plain}{\theorembodyfont{\rmfamily}
\newtheorem{example}[theorem]{Example}}
\theoremstyle{plain}{\theorembodyfont{\rmfamily}
\newtheorem{remark}[theorem]{Remark}}

\theoremstyle{plain}{\theorembodyfont{\rmfamily}
}

\def\hat{\widehat}

\def\bd{\mbox{\rm bd}\,}

\def\inte{\mbox{\rm int}\,}

\def\dim{\mbox{\rm dim}\,}
\def\dom{\mbox{\rm dom}\,}

\def\bd{\mbox{\rm bd}\,}

\newcommand{\RR}{\ensuremath{\mathbb R}}

\newcommand{\RX}{\ensuremath{\,\left]-\infty,+\infty\right]}}

\newcommand{\NN}{\ensuremath{\mathbb N}}

\newcommand{\di}{\ensuremath{\operatorname{dist}}}

\newcommand{\aff}{\operatorname{aff}}

  \newcommand{\qed}{\hspace*{\fill}$\Box$\medskip}
  \newcommand{\qede}{\hspace*{\fill}$\Diamond$\medskip}

\begin{document}
\title{\sffamily{Analysis of the convergence  rate for the cyclic projection algorithm applied to basic semi-algebraic convex sets }}

  \author{Jonathan M. Borwein\thanks{CARMA, University of Newcastle,
Newcastle, New South Wales 2308, Australia. E-mail:
\texttt{jonathan.borwein@newcastle.edu.au}. Laureate Professor at
the University of Newcastle and Distinguished Professor at  King
Abdul-Aziz University, Jeddah.},\; Guoyin Li \thanks{Department of Applied Mathematics, University of New South Wales,
Sydney 2052, Australia. E-mail: \texttt{g.li@unsw.edu.au}.},\;
and  Liangjin Yao\thanks{CARMA, University of Newcastle,
 Newcastle, New South Wales 2308, Australia.
E-mail:  \texttt{liangjin.yao@newcastle.edu.au}.} }

\date{Revised version: November 16 2013}
\maketitle
\begin{abstract} \noindent
In this paper, we study the rate of convergence of the cyclic projection algorithm applied to  finitely many basic
semi-algebraic convex sets.  We establish an explicit convergence rate estimate
 which relies on the maximum degree of the polynomials that generate the basic semi-algebraic convex sets and the dimension of the underlying space. We achieve our results by exploiting the algebraic structure of the basic semi-algebraic convex sets.

\end{abstract}

\noindent {\bfseries 2010 Mathematics Subject Classification:}\\
{Primary  41A25, 90C25;
Secondary 41A50, 90C31

}

\noindent {\bfseries Keywords:}
 Cyclic projection algorithm, convex polynomial, distance function,
Fej\'{e}r
monotone sequence, H\"{o}lderian  regularity, {\L}ojasiewicz's inequality,
projector operator, basic semi-algebraic convex set,
von
Neumann alternating projection method.

\section{Introduction}

A very common  problem in diverse areas of mathematics and engineering consists of trying to find a point in the intersection
of closed convex sets $C_i$, $i=1,\ldots,m$. This problem is often referred to as the convex feasibility problem. One popular
method for solving the convex feasibility problem is the so-called cyclic projection algorithm. Mathematically, the cyclic projection
algorithm is formulated as follows.
Given finitely many closed convex sets $C_1, C_2,\cdots, C_m$ in $\mathbb{R}^n$ with $\bigcap_{i=1}^m C_i\neq\varnothing$, let
 $x_0\in \mathbb{R}^n$ and
$P_i:=P_{C_i}, i=1,2,\cdots,m$, where $P_{C_i}$ denotes the Euclidean projection to the set $C_i$.
 The sequence of \emph{cyclic projections}, $(x_k)_{k\in\NN}$, is defined by
\begin{align}
x_1:=P_1 x_0, x_2:=P_2 x_1,\cdots, x_m:=P_mx_{m-1}, x_{m+1}:=P_1x_{m}\ldots\label{Cycse:1}
\end{align}
When $m=2$, the cyclic projection method  reduces to  the well known  von Neumann alternating projection method (APM) (see \cite{vonnum:1} and also \cite{Heinz_Set,BC2011,BLPW,KoRei} for some
recent developments). The cyclic projection method has attracted much recent interest  due to its simplicity and to numerous applications to diverse areas such as engineering and the physical sciences, see \cite{Bausc2, BaComt,Heinz, BrBSi1,Deus}
and the references therein.
 
 The convergence properties of cyclic projection methods have been examined by many researchers. In particular, Bregman \cite{Bregman} showed that the sequence $(x_k)_{k\in\NN}$ generated by the cyclic projection algorithm, always converges to a point in $C$.
Moreover, linear convergence of cyclic projection algorithm onto convex sets with regular intersections was shown in \cite{GPR}. On the other hand, for
convex sets with irregular intersections (for example, when the intersection is a singleton), the cyclic projection algorithm may not exhibit linear convergence even for simple two dimensional cases as observed by \cite[Example 5.3]{BBJAT1} (see Section~\ref{s:final} for more examples).
This then raises the following basic question:
\begin{quote}\emph{Can we estimate the convergence rate of the cyclic projection algorithm for
convex sets with possibly irregular intersections?}
\end{quote}

In this paper, we provide an answer for the above question by focusing on the case where  each set $C_i$ is a \emph{basic semi-algebraic convex set} in $\mathbb{R}^n$ in the sense that there exist $\gamma_i \in\NN$ and  convex polynomial functions, $g_{i\, j}, j=1,\ldots,\gamma_i$ such that
\begin{align*}C_i=\{x  \in \mathbb{R}^n\mid g_{i\, j}(x) \le 0, j=1,\cdots,\gamma_i\}.\end{align*}
The main motivation for examining  basic semi-algebraic convex sets lies with the following two facts. First, as recently established in \cite{Attouch1,Attouch2},  optimization problems involving semi-algebraic structure have a number of remarkable properties (such as the
celebrated Kurdyka-$\L$ojasiewicz inequality) which enables us to obtain useful qualitative information of the problem. Second, the class of basic semi-algebraic convex sets is a broad class of convex sets which covers polyhedra and convex sets described by convex quadratic functions. Additionally, the structure can often  be  relatively easily identified \cite{Nie}.

By exploiting the precise algebraic structure, we are able to provide an explicit rate for the cyclic projection algorithm applied to finitely many basic semi-algebraic convex sets \emph{without any regularity conditions}. More precisely, let $C_i$ be basic semi-algebraic convex sets generated by  polynomials in $\mathbb{R}^n$ with degree at most $d \in \mathbb{N}$.  
We show that the sequence of cyclic projections $(x_k)_{k\in\NN}$ \eqref{Cycse:1}  converges (at least) at the rate of
$\frac{1}{k^{\rho}}$ when $d>1$, where $\rho:=\frac{1}{\min\big\{(2d-1)^n-1,\, 2\beta(n-1)d^n-2\big\}}$ and $\beta(s)$ denotes the \emph{central binomial coefficient} with respect to $s$---which is given by ${s \choose {[s/2] }}$. \footnote{Here, $[a]$ denotes the integer part of $a$}  When $d=1$, the  sequence of cyclic projections converges linearly.

The remainder of this paper is organized as follows. In
Section~\ref{s:aux}, we collect notation and auxiliary results for future
use and for the reader's convenience.
In Section \ref{s:supphol}, we give a H\"{o}lderian  regularity result for  finitely many basic semi-algebraic convex sets.
 The proof of our main
result (Theorem~\ref{TheMTR:1}) forms the bulk of Section~\ref{s:main}. In Section~\ref{s:final}, we explore various concrete examples. Finally, we
end the paper with some conclusions and open questions.

\setcounter{equation}{0}
\section{Preliminaries and auxiliary results}\label{s:aux}

We assume throughout that $\mathbb{R}^n$ is a Euclidean space with the norm $\|\cdot\|$ and  inner product $\langle\cdot,\cdot\rangle$, where $n\in \NN:=\{1,2,3,\cdots\}$. We reserve $d\in\NN$. We denote by $\mathbb{B}(x,\varepsilon):=\{y\in\RR^n\mid\|y-x\|<\varepsilon\}$.
 We adopt standard notation used in these
books \cite{BC2011,BorVan, ph,
RockWets,Zalinescu}.

 Given a subset $C$ of $\RR^n$, $\inte C$ is the
\emph{interior} of $C$, $\bd{C}$ is the \emph{boundary} of $C$,
$\aff C$ is the \emph{affine hull} of $C$ and $\overline{C}$ is the
norm \emph{closure} of $C$. The \emph{orthogonal set} is $C^{\bot}:= \{x^*\in \RR^n
\mid(\forall c\in C)\, \langle x^*, c\rangle=0\}$. The \emph{distance function}  to   the set $C$, written  $\di(\cdot,C)$,
is defined by
$x\mapsto \inf_{c\in C}\|x-c\|$. The \emph{projector operator}  to   the set $C$, denoted by $P_C$,
is defined by
\begin{align*}P_C(x):=\{c\in C\mid \|y-c\|=\di(x, C)\},\quad\forall x\in \RR^n.
\end{align*} Let $D\subseteq\RR^n$. The distance of two sets: $C$ and $D$,
 is $\di(C,D):=\inf_{c\in C, d\in D}\|c-d\|$.
Given $f\colon X\to \RX$, we set
$\dom f:= f^{-1}(\RR)$. We say $f$ is \emph{proper} if $\dom f\neq\varnothing$.
Let $f$ be a proper function on
$\mathbb{R}^n$. Its associated \emph{recession function} $f^{\infty}$ is
defined, for any $v \in \mathbb{R}^n$, by
\begin{equation*} f^{\infty}(v):=\liminf_{t \rightarrow \infty,
\, v'\rightarrow v}\frac{f(tv')}{t}.
\end{equation*}
If $f$ is further assumed to be lower semicontinuous and convex, one
has (see \cite[Proposition 2.5.2]{Auslender})
\begin{equation} \label{eq:convex_recession_function}
f^{\infty}(v)=\lim_{t \rightarrow
\infty}\frac{f(x+tv)-f(x)}{t}=\sup_{t>0}\frac{f(x+tv)-f(x)}{t}\quad
\mbox{ for all } x\in {\rm dom}f.
\end{equation}

\subsection{Notation and facts on polynomials}
Recall that $f\colon\mathbb{R}^n\rightarrow\mathbb{R}$ is a {\em polynomial} if there exists a number $r\in\mathbb{N}$ such that
$$
f(x):=\sum_{0\le|\alpha|\le r}\lambda_{\alpha}x^{\alpha},
$$
where $\lambda_{\alpha}\in \mathbb{R}$, $x=(x_1,\cdots,x_n)$, $x^{\alpha}:=x_1^{\alpha_1}\cdots x_n^{\alpha_n}$, $\alpha_i\in\mathbb{N}\cup \{0\}$, and $|\alpha|:=\sum_{j=1}^{n}\alpha_j$.
The corresponding constant $r$ is called the {\em degree} of $f$.

Next let us recall a useful property of polynomial functions.
\begin{fact}\emph{(See \cite[Remark 4]{Auslender2000})}\label{lemma:3.1}
Let $f:\RR^n\rightarrow\RR$ be  polynomial,  and \{$x_1,x_2\}\subseteq\RR^n$. If $f$ is constant on
$D:=[x_1,x_2]$, then $f$ is constant on ${\rm aff}D$.
\end{fact}

We now summarize some basic properties of convex polynomials that will be used later. The first property is a Frank-Wolfe type result for convex polynomial optimization
problems while the second one is a directional-constancy property for a convex polynomial.
\begin{fact} [Belousov]\label{lemma:2.4}\emph{(See \cite[Theorem~13, Section 4, Chapter II]{Belousov} or \cite[Theorem~3]{convex_polynomial} and \cite{quasi_convex}.)}
Let $f$ be a convex polynomial on $\mathbb{R}^{n}$.  Consider a set $D:=\{x\mid g_i(x) \le 0, i=1,\cdots,m\}$, where each $g_i$, $i=1,\cdots,m$, is a convex polynomial on $\mathbb{R}^n$. Suppose that $\inf_{x \in D}
f(x)>-\infty$. Then $f$ attains its minimum on $D$.
\end{fact}

\begin{fact} \label{Factinft:1}\emph{(See \cite[Proposition~3.2.1]{Auslender}.)}
    Let $f$ be a convex polynomial on $\mathbb{R}^{n}$ and $v\in\RR^n$.
Assume that $f^{\infty}(v)=0$. Then $f(x+tv)=f(x)$ for all $t \in \mathbb{R}$ and for all $x \in \mathbb{R}^n$.
\end{fact}

\subsection{Notation and facts on semialgebraic sets/functions}
Following \cite{real},
a set $D \subseteq \mathbb{R}^n$ is said to be
\emph{semi-algebraic} if
$$D: =\bigcup_{i=1}^{l} \, \bigcap_{j=1}^{s}\{x \in \RR^n\mid f_{ij}(x)=0, h_{ij}(x)\leq0\}$$
for some integers $l,s$ and some polynomial functions $f_{ij}, \,
h_{ij}$ on $\mathbb{R}^{n}$ $(1 \le i\le l, \, 1 \le j \le s)$.
Moreover, a function $f:\mathbb{R}^{n} \rightarrow \mathbb{R}$ is
said to be \emph{semi-algebraic} if its \emph{graph} ${\rm gph}f:=\{(x,f(x))\mid x \in
\mathbb{R}^n\}$ is semi-algebraic.

We now summarize below some basic properties of semi-algebraic sets and
semi-algebraic functions. These properties will be useful for our later  work.

\begin{fact} The following statements hold (the properties (P1) and (P4) are direct from the definitions).
\begin{enumerate}
\item[\text{(P1)}] Any polynomial is a semi-algebraic function.

\item[\text{(P2)}] \emph{(See \cite[Proposition 2.2.8]{real}.)} Let $D$ be a semi-algebraic
set. Then $\di(\cdot,D)$ is a
semi-algebraic function.
\item[\text{(P3)}] \emph{(See \cite[Proposition 2.2.6]{real}.)} If $f,\, g$ are
semi-algebraic functions on $\mathbb{R}^n$ and $\lambda \in \mathbb{R}$
then $f+g$, $\lambda f$, $\max\{f,g\}$, $fg$ are  semi-algebraic.
\item[\text{(P4)}]  If $f_i$ are polynomials, $i=1,\ldots,m$,
 and $\lambda \in \mathbb{R}$, then the sets $\{x\mid f_i(x)=\lambda, i=1,\ldots,m\}$,
 $\{x\mid f_i(x) \le \lambda, \, i=1,\ldots,m\}$ are
semi-algebraic sets.
\item[\text{(P5)}] (\rm {\L}ojasiewicz's inequality)\emph{(See \cite[Corollary~2.6.7]{real}.)} If
$\phi,\psi$ are two continuous semi-algebraic functions on
compact semi-algebraic set $K \subseteq \mathbb{R}^n$ such that
$\emptyset \neq \phi^{-1}(0) \subseteq \psi^{-1}(0)$ then there
exist constants $c>0$ and $\tau \in (0,1]$ such that
\[
 |\psi(x)| \le  c |\phi(x)|^{\tau}\quad \mbox{ for all } x \in K.
\]
\end{enumerate}
\end{fact}

\begin{remark} \label{remark:2} As pointed out by \cite{Luo}, the corresponding exponent  $\tau$ in {\L}ojasiewicz's
inequality (P5) is hard to determine and is typically not known. \qede
\end{remark}

\begin{remark} \label{re:semialgebraic}
Let $g_i$, $i=1,\ldots,m$ be  polynomials on $\RR^n$ and set  $S:=\{x\in\RR^n\mid g_i(x) \le 0\}$. Let $\bar x \in S$. Then, (P2) and (P4) imply that
$\phi=\max_{1 \le i \le m}[g_i]_+$ while $[g_i]_+:=\max\{g_i(\cdot),0\}$  and $\psi=\di(\cdot,S)$ are semi-algebraic functions. Applying (P5) it follows
 that there exist $c,\varepsilon
>0$ and $\tau \in (0,1]$ such that
\begin{equation}\label{eq:0088}
\di(x,S) \le c \max_{1 \le i \le m}[g_i(x)]_+^{\tau}\quad \mbox{ for all
} x \in \mathbb{B}(\bar x, \varepsilon).
\end{equation}\qede
\end{remark}

As  explained in Remark \ref{remark:2}, the exponent $\tau$ in \eqref{eq:0088} is hard to determine and is typically unknown. However, there are some special cases where
we can provide some effective estimates on the exponent $\tau$:   To formulate these results, we  introduce the following notation.
Define \begin{equation}\label{eq:kappa}
\kappa(n,d):=(d-1)^n+1.
\end{equation}

We now present various results which show that the exponent $\tau$ in \eqref{eq:0088} can be effectively estimated  when $g_i$ has some appropriate extra
structure.

\begin{fact} [Gwo\'{z}dziewicz]\label{lemma:3.2}\emph{(See \cite[Theorem 3]{effective}.)}
  Let  $g$ be a polynomial on $\mathbb{R}^n$ with degree no larger than $d$.
 Suppose that $g(0)=0$ and there exists $\varepsilon_0>0$ such that $g(x)>0$ for all $x \in \mathbb{B}(0,\varepsilon_0)\backslash \{0\}$.  Then there exist constants $c,\varepsilon>0$ such
that
\begin{equation} \label{eq:loja}
\|x\| \le  c \, g(x)^{\frac{1}{\kappa(n,d)}},\quad \forall x \in
\mathbb{B}(0,\varepsilon).
\end{equation}
\end{fact}

 We  denote by $\beta(s)$  the \emph{central binomial coefficient} with respect to an integer $s$: ${s \choose {[s/2] }}$ (with $\binom {0} {0}=1$)\cite{Kollar}.
\begin{fact} [Koll\'{a}r]\label{KolThe:1}\emph{(See \cite[Theorem~3(i)]{Kollar}.)}
 Let $g_i$  be polynomials on $\mathbb{R}^n$ with degree  $\leq d$ for every $i=1,\cdots,m$.  Let $g(x):=\max_{1\leq i\leq m} g_i(x)$.
 Suppose that there exists $\varepsilon_0>0$ such that $g(x)>0$ for all $x \in \mathbb{B}(0,\varepsilon_0)\backslash \{0\}$.  Then there exist constants $c,\varepsilon>0$ such
that
\begin{equation*}
\|x\| \le  c \, g(x)^{\frac{1}{\beta(n-1)d^n}},\quad \forall x \in
\mathbb{B}(0,\varepsilon).
\end{equation*}
\end{fact}

\begin{fact} \label{lemma:li0}\emph{(See \cite[Theorem~4.2]{Li0}.)}
 Let $g$ be a convex polynomial on $\mathbb{R}^n$ with degree at most $d$. Let $S:=\{x\mid g(x) \le 0\}$ and  $\bar x \in S$. Then, $g$ has a H\"{o}lder type local error bound with exponent $\kappa(n,d)^{-1}$, i.e.,
 there exist  constants $c,\varepsilon>0$ such
that
\[
\di(x,S) \le c \, [g(x)]_{+}^{\frac{1}{\kappa(n,d)}}\quad \mbox{ for all } x \in \mathbb{B}(\bar x, \varepsilon).
\]
\end{fact}

\begin{fact} \label{lemma:li1}\emph{(See \cite[Theorem~4.1]{Li0}.)}
 Let $g$ be a convex polynomial on $\mathbb{R}^n$. Let $S:=\{x\mid g(x) \le 0\}$. Suppose that there exists $x_0\in \mathbb{R}^n$ such that $g(x_0)<0$. Then, $g$ has a Lipschitz type global error bound, i.e.,
 there exists  a constant $c>0$ such
that
\[
\di(x,S) \le c \, [g(x)]_{+} \quad \mbox{ for all } x \in \mathbb{R}^n.
\]
\end{fact}

\begin{fact}\label{lemma:Robinson}\emph{(See \cite[Theorem 2]{Robinson}}
 Let $g$ be a continuous convex function $\RR^n$.  Let $S:=\{x \in \mathbb{R}^n\ \mid g(x) \le 0\}$. Suppose that there exists $x_0\in \mathbb{R}^n$ 
  such that $g(x_0)<0$. Then, for every compact subset $K$  of $\RR^n$,
  there exists  $c>0$ such
 that
 \[
 \di(x,S) \le c \, [g(x)]_{+} \quad \mbox{ for all } x \in K.
 \]
\end{fact}

The following example  show us that the conclusion of Fact \ref{lemma:Robinson} can fail if we allow $K$ to be noncompact.

\begin{example}[Shironin]\emph{\rm( See \cite[Example~4.1]{Li0} or \cite{Shironin}.)}
Let $g_1, g_2: \RR^4\rightarrow\RR$ be defined by
\begin{align*}
g_1(x_1, x_2,x_3, x_4)&:=x_1,\\
g_2(x_1, x_2,x_3, x_4)&:=x_1^{16}
+ x_2^{8}
+ x_3^{6}
+ x_1 x^3_2x_3^3+x_1^2x_2^4x_3^2+ x_2^2
x_3^4+x_1^4x_3^4\\
&\quad+x_1^4x_2^6+x_1^2x_2^6+x_1^2+x_2^2+x_3^2-x_4.
\end{align*}
Then $g_1$,$g_2$ are convex polynomials and
\begin{align*}
g_1(-k,0,0,k^{16}+k^2+k) = g_2(-k,0,0,k^{16}+k^2+k)=-k<0,\quad\forall k\in\NN.
\end{align*}
Let $g(x):=\max\{g_1(x),g_2(x)\}$ for all $x \in \mathbb{R}^4$. Then, $g$ is a continuous convex function and $g(-k,0,0,k^{16}+k^2+k)<0$ for all $k \in \NN$.
On the other hand, as shown in \cite{Shironin}, there exists a sequence $(x_k)_{k\in\NN}$ in $\RR^4$ such that
\begin{align*}
[g(x_k)]_+ \leq 1,\,\forall k\quad\text{but}\quad \di(x_k, S)\longrightarrow +\infty,
\end{align*}
where $S:=\{x\in\RR^4\mid g_1(x)\leq0, g_2(x)\leq 0\}$. Let $K$ be a noncompact set such that $\{x_k\} \subseteq K$. Then, the conclusion of Fact \ref{lemma:Robinson} 
``$\di(x,S) \le c \, [g(x)]_{+} \quad \mbox{ for all } x \in K$'' must fail in this case.
\qede
\end{example}

Recall that a set $C \subseteq \mathbb{R}^n$ is  a \emph{basic semi-algebraic convex} set if there exist $\gamma \in\NN$ and  convex polynomial functions, $g_{j}, j=1,\ldots,\gamma$ such that $C=\{x  \in \mathbb{R}^n\mid g_{j}(x) \le 0, j=1,\cdots,\gamma\}.$ Clearly, any basic semi-algebraic convex set is convex and semi-algebraic. However, the following example shows that a convex and semi-algebraic set need not to be a basic semi-algebraic convex set.

\begin{example} \label{remark:1}
Consider the set $A:=\{(x_1,x_2)\in\RR^2\mid 1-x_1x_2 \le 0, -x_1 \le 0, -x_2 \le 0\}$. Clearly, $A$ is convex and semi-algebraic while the polynomial $(x_1,x_2)\mapsto 1-x_1x_2$ is not convex. We now show that $A$ is not a basic semi-algebraic convex set, i.e., it cannot be written as $\{x: g_i(x) \le 0, i=1,\cdots,l\}$ for some convex polynomials $g_i$, $i=1,\cdots,l$, $l \in \mathbb{N}$. To see this,
we proceed by the method of contradiction. Let $f:\mathbb{R}^2 \rightarrow \mathbb{R}$ be defined by $f(x_1,x_2):=x_1$. Clearly $\displaystyle \inf_{x=(x_1,x_2)\in A}f(x) =0$. Then,  then by Fact~\ref{lemma:2.4}, $f$ should attain its minimum on $A$. This leads to a  contradiction, and so,  justifies the claim. \qede
\end{example}

\subsection{Notation and facts for projection methods}

From now on, we assume that
\begin{equation*}
   \addtolength{\fboxsep}{5pt}
   \boxed{
   \begin{alignedat}{5}
   &m\in\NN, \gamma_i\in\NN,\, i=1,\ldots,m\\
     & g_{i,1}, g_{i,2},\cdots, g_{i, \gamma_i}\,\text{ are convex polynomials  on}\,\RR^n, \, i=1,\ldots,m\\
&
C_i:=\Big\{x\in\RR^n\mid g_{i,1}(x)\leq0, g_{i,2}(x)\leq0,\cdots, g_{i,\gamma_i}(x)\leq 0  \Big\},\, i=1,\ldots,m\\
&P_i:=P_{C_i},\quad \forall i=1,2,\cdots,m\\
&C:=\bigcap_{i=1}^m C_i\neq\varnothing.
\end{alignedat}
   }
\end{equation*}

\begin{fact}[Bregman]\emph{(See \cite{Bregman}.)}\label{BregFms:1}
Let $x_0\in\RR^n$. The sequence of cyclic projections, $(x_k)_{k\in\NN}$, defined by
\begin{align}
x_1:=P_1 x_0, x_2:=P_2 x_1,\cdots, x_m:=P_mx_{m-1}, x_{m+1}:=P_1x_{m}\ldots
\end{align}   converges to a point in $C$.
\end{fact}

\begin{fact}[Bauschke and Borwein]\emph{(See \cite[Lemma~2.2 and Theorem~4.8]{BBJAT1}, or \cite[Fact~1.1(iii) and Fact~1.2(ii)]{Heinz_Set}.)}\label{FactPr:1}
Let $A, B$ be nonempty convex subsets of $\RR^n$ such that $A- B$ is closed. Let $b_0\in X$ and $(a_k)_{k\in\NN}, (b_k)_{k\in\NN}$ be defined as below:
\begin{align*}
\begin{cases}
a_{k+1}&:=  P_Ab_k \\
b_{k+1}&:=  P_Ba_{k+1}.
\end{cases}
\end{align*}
Let $v:=P_{A-B}0$. Then, we have \begin{enumerate}
 \item[\rm (i)] \label{FactPr:a}$\|v\|=\di(A,B)$ and
$a_k\longrightarrow a,\quad b_k\longrightarrow a+v$.

\item[\rm (ii)]\label{FactPr:b}  $P_Bx=P_{B\cap(A+v)}x=x+v,\,\forall x\in A\cap (B-v)$\quad and \quad
$P_Ay=P_{A\cap (B-v)}y=y-v,\,\forall y\in B\cap(A+v)$
\end{enumerate}
\end{fact}

\begin{definition}
Let $A$ be a nonempty convex subset of $\RR^n$. We say the sequence $(x_k)_{k\in\NN}$ in
$\RR^n$ is \emph{Fej\'{e}r
monotone} with respect to $A$ if
\begin{align*}
\|x_{k+1}-a\|\leq\|x_k-a\|,\quad\forall k\in\NN,\, a\in A.
\end{align*}

\end{definition}

\begin{fact}[Bauschke and Borwein]\emph{(See \cite[Theorem~3.3(iv)]{Heinz_Set}.)}\label{FactPr:2}
Let $A$ be a nonempty closed  convex subset of $\RR^n$ and let $(x_k)_{k\in\NN}$ be Fej\'{e}r
monotone with respect to $A$, and $x_k\longrightarrow x\in A$. Then $\|x_n-x\|\leq2\di(x_n,A)$.
\end{fact}

\setcounter{equation}{0}
\section{H\"{o}lderian  regularity for basic semi-algebraic convex sets}\label{s:supphol}

In this section,  we will establish H\"{o}lderian  regularity for basic semi-algebraic convex sets and shall provide an effective estimate of the exponent in the regularity results.
This result plays an important role in our following estimation of the convergence speed of the cyclic projection methods.

To do this, we first establish an error bound result which estimates the distance of a point to a basic semi-algebraic convex set $S$ in terms of
the polynomials which define $S$. More explicitly, we obtain an  explicit exponent $\tau>0$ such that there exist $c,\varepsilon>0$,
\[
\di(x,S)\le c\bigg(\max_{1 \le i \le m} [g_i(x)]_+\bigg)^{\tau}\quad  { whenever } \quad \|x-\overline{x}\|\le\epsilon,
\]
where $S:=\{x \in \mathbb{R}^n\mid g_i(x) \le 0, i=1,\cdots,m\}$. 

We note that
this error bound property plays an important role in convergence analysis of many  algorithms for optimization problems \cite{Pang_MP} as well
as  for the variational inequality problem \cite{Li1}, and that the exponent $\tau$ in the error bound property has a close relationship with
 the convergence rate of the algorithm. However,  existing results such as the powerful {\L}ojasiewicz's inequality do not provide any insight on how to explicitly estimate the exponent $\tau$.

Before we proceed, let us use a simple example to illustrate that the exponent $\tau$ can be related to the maximum degree of the polynomials defined the basic semi-algebraic convex set and the dimension of the underlying space. This example is partially inspired by \cite[Example~1]{Kollar}.

\begin{example}\label{ex:01}
Let $d$ be an even number. Consider convex polynomials $g_i$, $i=1,\ldots,n$ on $\mathbb{R}^n$ given by  $g_1(x):=x_1^d$ and $g_i$, $i=2,\cdots,n$ given by
$g_i(x):=x_{i}^d-x_{i-1}$, $i=2,\cdots,n$. Then, direct verification gives us that $S:=\{x \in \mathbb{R}^n\mid g_i(x) \le 0, i=1,\cdots,n\}=\{0\}$, and so,
$\di(x,S)=\|x\|$. In this case, consider $x(t)=(t^{d^{n-1}},t^{d^{n-2}},\cdots,t) \in \mathbb{R}^n$, $t \in (0,1)$. Then $\di(x(t),S)=O(t)$ and $\max_{1 \le i \le m}[g_i(x(t))]_+=t^{d^n}$. Therefore, we
see that if there exist $c,\epsilon>0$ and $\tau>0$ such that \[
\di(x,S)\le c\bigg(\max_{1 \le i \le m} [g_i(x)]_+\bigg)^{\tau}\quad  { whenever } \quad \|x\|\le\epsilon,
\]
then, $\tau \le \frac{1}{d^{n}}$. Thus, we see that the exponent $\tau$ is related to the maximum degree  of the polynomials defined the basic semi-algebraic convex set and the dimension of the underlying space. \qede
\end{example}

We now introduce a decomposition of the index set.
\begin{definition}
For convex polynomials $g_1,\ldots,g_m$ on $\mathbb{R}^n$ with $S=\{x\in\RR^n\mid g_i(x) \le 0, i=1,\ldots,m\}$, the index set $\{1,\ldots,m\}$ can be decomposed as $J_0 \cup J_1$ with $J_0 \cap J_1=\emptyset$ where
\begin{equation}\label{eq:J1}
J_0:=\big\{i \in \{1,\ldots,m\}\mid g_i(S)\equiv0 \big\} \quad\mbox{ and } \quad J_1:=\{1,\ldots,m\} \backslash J_0.
\end{equation}
\end{definition}

Now we come to our key technical result which provides a local error bound for convex polynomial systems. The main idea of the proof is to use the extreme rigidity of convex polynomial functions (see Fact \ref{lemma:3.1} and \ref{Factinft:1}) to reduce the
problem to the one of comparing a (power of a) nonnegative convex polynomial
vanishing only at 0 with the norm so that the estimate  of the Lojasiewicz inequality
that we introduced in Fact \ref{lemma:3.2} and Fact \ref{KolThe:1} can be applied.

\begin{theorem}{\bf (Local error bounds for convex polynomial systems)} \label{Tprop:2.0} Let $g_i$  be convex polynomials on $\mathbb{R}^n$ with degree at most $d$ for every $i=1,\cdots,m$. Let $S:=\{x \in \mathbb{R}^n\ \mid g_i(x) \le 0, i=1,\cdots,m\}$ and $\bar x \in S$. Then there exist $c, \varepsilon > 0$ such that
\[
\di(x,S)\le c\bigg(\max_{i \in J_1}[g_i(x)]_+ + \big(\max_{i \in J_0} [g_i(x)]_+\big)^{\tau}\bigg)\quad  { whenever } \quad \|x-\overline{x}\|\le\varepsilon,
\]
where $[a]_+:=\max\{a,0\}$, $\tau:=\max\big\{\frac{2}{\kappa(n,2d)}, \frac{1}{\beta(n-1)d^n}\big\}$, $\kappa(n,2d):=(2d-1)^n+1$, $\beta(n-1)$ is the central binomial coefficient with respect to $n-1$ which is given by ${n-1 \choose {[(n-1)/2] }}$, and $J_0, J_1$ are defined as in (\ref{eq:J1}).
\end{theorem}
\begin{proof}
We prove the desired conclusion by induction on the number of the polynomials $m$.

{\bf [Trivial Case]} Suppose that $m=1$. Then $J_0=\{1\}$ or $J_0=\emptyset$. If $J_0=\{1\}$, then the conclusion follows by Fact~\ref{lemma:li0} since $\max\big\{\frac{2}{\kappa(n,2d)}, \frac{1}{\beta(n-1)d^n}\big\}\leq \frac{1}{\kappa(n,d)}$. If $J_0=\emptyset$, then there exists $x_0$ such that $g_1(x_0)<0$. In this case, the conclusion follows
by Fact~\ref{lemma:li1}.

{\bf [Reduction to the active cases]} Let us suppose that the conclusion is true for $m \le p-1$, $p \in \mathbb{N}$, and look at the case for $m=p$.  If $J_0 \neq \{1,\ldots,m\}$, then $\{1,\ldots,m\} \backslash J_0 \neq \emptyset$.
Let $i_0\notin J_0$. Then there exists $x_0\in S$ such that $g_{i_0}(x_0) < 0$.
 Set $J:=\big\{i\in\{1,2,\ldots,m\}\mid g_i(x_0)<0\big\}$. Then $i_0\in J\subseteq J_1$ and $J\cap J_0=\varnothing$.
Let $A,B$ be defined by \begin{align*}
A&:=\big\{x\in\RR^n\mid g_i(x)\leq0,\quad \forall i\in J\big\}\\
B&:=\big\{x\in\RR^n\mid g_i(x)\leq0,\quad \forall i\in \widetilde{J}\big\},
\end{align*}where $\widetilde{J}:=\{1,2,\ldots,m\}\backslash  J$.
 Thus, $J_0\subseteq \widetilde{J}$. Then we have $x_0\in\inte A\cap B$ and $S=A\cap B$.  Since $S\subseteq B$, we have
 \begin{align}\label{ExLSC:1}
 \widetilde{J_0}:=\big\{i\in \{1,2,\ldots,m\}\backslash  J\mid g_i(B)\equiv0\big\}\subseteq J_0.
 \end{align}
Since $\inte A\cap B \neq \emptyset$, \cite[Corollary~4.5]{Heinz_Set} implies that for every compact set $K$  there exist $\gamma,\delta>0$ such that
\begin{align}
\di(x, S)=\di(x,A \cap B) \le \gamma \max\{\di(x,A),\di(x,B)\} \quad \mbox{for all } x \in K.\label{Tspple:rc1}
\end{align}
Now applying Fact~\ref{lemma:Robinson} with $f(x)=\max_{i \in J}[g_i(x)]_+$,  there exists $c_1>0$ such that
\begin{align}
\di(x,A) \le c_1 \, \max_{i \in J}[g_i(x)]_+ ,\quad\forall x\in K.\label{Tspple:rc2}
\end{align}
\allowdisplaybreaks
From the induction hypothesis and \eqref{ExLSC:1}, we see that there exist $\varepsilon>0$ and $c_2>0$ such that
for every $\|x-\bar x\|\le \varepsilon$, $\max_{1\leq i\leq m }[g_i(x)]_+\leq 1$ such that
\begin{align*}
&\di(x,B)\\
&\le c_2\bigg(\max_{i \in \widetilde{J}\backslash  \widetilde{J}_0}[g_i(x)]_+ + \big(\max_{i \in \widetilde{J_0}} [g_i(x)]_+\big)^{\tau}\bigg) \\
&= c_2\bigg(\max\big\{\max_{i \in \widetilde{J}\backslash  J_0}[g_i(x)]_+, \max_{i \in J_0\backslash \widetilde{J}_0}[g_i(x)]_+ \big\}+ \big(\max_{i \in \widetilde{J_0}} [g_i(x)]_+\big)^{\tau}\bigg) \quad\text{(since $\widetilde{J_0}\subseteq J_0\subseteq \widetilde{J}$)}\\
&\leq c_2\bigg(\max_{i \in \widetilde{J}\backslash  J_0}[g_i(x)]_+ +\max_{i \in J_0\backslash \widetilde{J}_0}[g_i(x)]_+ + \big(\max_{i \in \widetilde{J_0}} [g_i(x)]_+\big)^{\tau}\bigg) \\
&\leq c_2\bigg(\max_{i \in \widetilde{J}\backslash  J_0}[g_i(x)]_+ +\big(\max_{i \in J_0\backslash \widetilde{J}_0}[g_i(x)]_+\big)^{\tau} + \big(\max_{i \in \widetilde{J_0}} [g_i(x)]_+\big)^{\tau}\bigg) \\
&\leq
c_2\bigg(\max_{i \in \widetilde{J}\backslash  J_0} [g_i(x)]_++ 2\big(\max_{i \in J_0} [g_i(x)]_+\big)^{\tau}\bigg)\\
&\leq 2c_2 \bigg(\max_{i \in \widetilde{J}\backslash  J_0} [g_i(x)]_++ \big(\max_{i \in J_0} [g_i(x)]_+\big)^{\tau}\bigg).
\end{align*}
Thus, the conclusion follows in this case  by combining \eqref{Tspple:rc1} and \eqref{Tspple:rc2}, and noting that $(\widetilde{J}\backslash  J_0) \cup J \subseteq J_1$.

 From now on, we may assume that $J_0=\{1,\ldots,m\}$. That is,
\[
\big\{x\mid g_i(x) \le 0,i=1,\cdots,m\big\}=\big\{x\mid g_i(x) = 0,i=1,\cdots,m\big\}.
\]
This implies that $\inf_{x \in \mathbb{R}^n} \max_{1 \le i \le m}\{g_i(x)\}=0$. Then, $$0_{\mathbb{R}^m} \notin \big\{(g_1(x),\cdots,g_m(x))
\mid x \in \mathbb{R}^n\big\}+\inte\mathbb{R}^m_+.$$
Hence the convex separation theorem ensures that there exist $\alpha_i \ge 0$ with $\sum_{i=1}^m\alpha_i=1$ such that
$\sum_{i=1}^m \alpha_i g_i(x) \ge 0 \mbox{ for all } x \in \mathbb{R}^n$.
Denote
$I:=\{i\mid \alpha_i >0\} \neq \emptyset$.
Then, we have $\sum_{i \in I}\alpha_i=1$ and
\begin{equation}\label{eq:66}
\sum_{i \in I} \alpha_i g_i(x) \ge 0 \quad\mbox{ for all } x \in \mathbb{R}^n.
\end{equation}
{\bf [Decompose the underlying space into sum of two subspaces $M$ and $M^{\bot}$]} Consider $D:=\{x\in\RR^n\mid g_i(x) \le 0, i \in I\}$. Clearly, $D$ is a convex set and $\bar x \in D$. Moreover, for any $v \in D$, (\ref{eq:66}) implies that
\begin{align}
g_i(v)=0, \quad \forall i \in I.\label{E2sump:1}
\end{align}
In other words, $g_i$ takes constant value $0$ on $D$. Then, Fact~\ref{lemma:3.1} implies that $D$ is either a singleton or an affine set with dimension larger than one.

Let $M:=D-\bar x$. Then $M$ is a subspace. We may decompose $\mathbb{R}^n=M +M^{\bot}$. Denote ${\rm dim}M=k$ $(k \le n)$.

We now see that
\begin{equation}\label{eq:003}
\sum_{i \in I} g_i^2(x)>0 \mbox{ for all } x-\bar x \in M^{\bot} \backslash\{0\}.
\end{equation}
Otherwise, there exists $x_0\in\RR^n$ such that $x_0-\bar{x}\in M^{\bot}\backslash\{0\}$ and
$g_i(x_0)=0$ for all $i \in I$. This shows that $x_0 \in D$. Thus
 $x_0-\bar x \in M$ and hence $x_0-\bar x \in M\cap M^{\bot}$. This contradicts the fact that $x_0-\bar x \neq0$.

Similarly, we have
\begin{equation}\label{KolEeq:1}
\max_{i \in I} g_i(x)>0 \mbox{ for all } x-\bar x \in M^{\bot} \backslash\{0\}.
\end{equation}

{\bf [Distance estimation on $M^{\bot}$]} We first show that there exist $\varepsilon_0,\gamma_0>0$ such that
\begin{equation}\label{eq:002}
\|x-\bar x\| \le \gamma_0 \bigg(\sum_{i \in I} g_i^2(x)\bigg)^{\frac{1}{\kappa(n-k,2d)}},\quad\forall x-\bar x \in M^{\bot}\cap \mathbb{B}(0,\varepsilon_0).
\end{equation}
Since $\dim M^{\bot}=n-k$, there exists an $n\times (n-k)$ matrix $Q_0$ with the rank $n-k$ such that $Q_0(\RR^{n-k})=M^{\bot}$. Then $Q_0$ is a bijective operator from $\RR^{n-k}$ to $M^{\bot}$.
Then \eqref{eq:003} shows that
\begin{align}\label{E3sump:1}
\sum_{i \in I} g_i^2(\bar{x}+Q_0b)>0,\quad \forall b\in\RR^{n-k} \backslash\{0\}.
\end{align}
Define $h:\RR^{n-k}\rightarrow\RR$
by $h(b):=\sum_{i \in I}g_i^2(\bar{x}+Q_0b)$. Then $h(0)=\sum_{i \in I}g_i^2(\bar{x})=0$ by \eqref{E2sump:1}.  Thus by  \eqref{E3sump:1} and Fact~\ref{lemma:3.2},
there exist ${\varepsilon_1},{\gamma_1}>0$ such that for all
\begin{align*}
\|b\|\leq {\gamma_1} h(b)^{\frac{1}{\kappa(n-k,2d)}}={\gamma_1}\Big(\sum_{i \in I}g_i^2(\bar{x}+Q_0b)\Big)^{\frac{1}{\kappa(n-k,2d)}},\quad\forall \|b\|\leq{\varepsilon_1}.
\end{align*}
Setting $x:=\bar{x}+Q_0b$, it follows that
\begin{align*}\|x-\bar{x}\|&=\|Q_0\big(Q_0^{-1}(x-\bar{x})\big)\|\leq\|Q_0\|\cdot\|Q_0^{-1}(x-\bar{x})\|\\
&\leq  \|Q_0\| {\gamma_1}\Big(\sum_{i \in I}g_i^2(x)\Big)^{\frac{1}{\kappa(n-k,2d)}},\quad\forall \|Q_0^{-1}(x-\bar{x})\|\leq{\varepsilon_1}.
\end{align*}
Hence, there exist $\varepsilon_1,\gamma_1>0$ such that
\begin{equation*}
\|x-\bar x\| \le \gamma_0 \bigg(\sum_{i \in I} g_i^2(x)\bigg)^{\frac{1}{\kappa(n-k,2d)}},\quad\forall x-\bar x \in M^{\bot}\cap \mathbb{B}(0,\varepsilon_0) .
\end{equation*}
Thence \eqref{eq:002} holds.

By Fact~\ref{KolThe:1} and \eqref{KolEeq:1}, there exist
$\widetilde{\varepsilon_0},\widetilde{\gamma_0}>0$ such that
$\widetilde{\varepsilon_0}\leq \varepsilon_0$ and, for all $x$ with $x-\bar x \in M^{\bot}\cap \mathbb{B}(0,\widetilde{\varepsilon_0})$,
\[
\max_{i \in I} g_i(x) \leq 1
\]
and
\begin{align}
 \|x-\bar x\| &\le \widetilde{\gamma_0} \bigg(\max_{i \in I} g_i(x)\bigg)^{\frac{1}{\beta(n-k-1)d^{n-k}}}\nonumber\\
&=\widetilde{\gamma_0} \bigg(\max_{i \in I} [g_i(x)]_+\bigg)^{\frac{1}{\beta(n-k-1)d^{n-k}}} .\label{KolEeq:2}
\end{align}

{\bf [Distance estimation on $M$]}
Set $r:=\begin{cases}\max\{\frac{\sum_{j \in I \backslash \{i\}}\alpha_j}{\alpha_i}\mid i \in I\}>0,\quad \,&\text{if}\,|I|\geq2;\\
1,&\,\text{otherwise}.
\end{cases}$. Thus $r\geq1$. Note that $\sum_{i \in I}\alpha_i g_i(x) \ge 0$, we have for each $i \in I$
\[
\max_{i \in I}[g_i(x)]_+ \ge g_i(x) \ge -\frac{\sum_{j \in I \backslash \{i\}}\alpha_j g_j(x)}{\alpha_i} \ge -r \max_{i \in I}[g_i(x)]_+\ .
\]
Hence we have $|g_i(x)| \le r \max_{i \in I}[g_i(x)]_+$. This together with (\ref{eq:002}) implies that
\begin{equation*}
\|x-\bar x\| \le  \gamma_0 r^2 |I| \bigg(\max_{i \in I}[g_i(x)]_+\bigg)^{\frac{2}{\kappa(n-k,2d)}},\quad\forall x-\bar x \in M^{\bot}\cap \mathbb{B}(0,\varepsilon_0).
\end{equation*}
Combining this with \eqref{KolEeq:2}, we see that, for every $x-\bar x \in M^{\bot}\cap \mathbb{B}(0,\widetilde{\varepsilon_0})$,
\begin{equation}\label{eq:0066}
\|x-\bar x\| \le  \big(\gamma_0 r^2 |I|+\widetilde{\gamma_0} \big) \bigg(\max_{i \in I}[g_i(x)]
_+\bigg)^{\max\{\frac{2}{\kappa(n-k,2d)},\frac{1}{\beta(n-k-1)d^{n-k}}\}}.
\end{equation}

We now consider two cases.

\emph{Case 1}: $\dim M=\{0\}$.

 We have $D=S=\{\bar{x}\}$.
Thus $M=0$ and $M^{\bot}=\RR^n$.
We can assume that $\max_{i\in I}[g_i(x)]_+ \le 1$ for all $x-\bar x
\in\mathbb{B}(0,\widetilde{\varepsilon_0})$.
Then by \eqref{eq:0066},
we have
\begin{align*}
\di(x,S)&=\|x-\bar x\| \le  \big(\gamma_0 r^2 |I|+\widetilde{\gamma_0} \big) \bigg(\max_{i \in I}[g_i(x)]_+\bigg)^{\max\big\{\frac{2}{\kappa(n,2d)},
\frac{1}{\beta(n-1)d^{n}}\big\}},\quad \forall \|x-\overline{x}\|\le\widetilde{\varepsilon_0}.
\end{align*}

\emph{Case 2}:  $k=\dim M\geq 1$.

Since $\dim M=k$, there exists a full rank matrix $Q \in \mathbb{R}^{n \times k}$ such that $Q(\mathbb{R}^k)=M$.  For each $u \in M$ and $i \in I$, \eqref{E2sump:1} implies that
\[
g_i^{\infty}(u)=\lim_{t \rightarrow \infty}\frac{g_i(\bar x+tu)-g_i(\bar x)}{t}=0.
\]
 Then, Fact~\ref{Factinft:1} implies that
\begin{equation} \label{eq:directional_constant}
g_i(x+u)=g_i(x) \quad\mbox{ for all } x \in \mathbb{R}^n, u \in M, i\in I.
\end{equation}
Since $S\subseteq D=\bar x+M$, it follows that
\[
S=\{x \in \bar x+M \in \mathbb{R}^n\mid g_i(x) \le 0, i \notin I\}=\bar x+Q(\hat{S}),
\]
where $\hat{S}:=\{a \in \mathbb{R}^k\mid g_i(\bar x+Qa) \le 0, i \notin I\}$.

Note that $0 \in \hat{S}$.
The induction hypothesis implies that there exist $\widetilde{\varepsilon_1},\widetilde{\gamma_1}>0$ such that $\max_{i \notin I}[g_i(\bar x+Qa)]_+\leq1$ and
\[
\di(a,\hat{S}) \le \widetilde{\gamma_1} \Big(\max_{i \notin I}[g_i(\bar x+Qa)]_+\Big)^{\max\big\{\frac{2}{\kappa(k,2d)},\frac{1}{\beta(k-1)d^k}\big\}} \quad\mbox{ for all } \|a\| \le \widetilde{\varepsilon_1}.
\]
This implies that there exist $\epsilon_2,\gamma_2>0$ such that
\begin{equation}\label{eq:001}
\di(x,S) \le \gamma_2 \Big(\max_{i \notin I}[g_i(x)]_+\Big)^{\max\big\{\frac{2}{\kappa(k,2d)},\frac{1}{\beta(k-1)d^k}\big\}}\quad \mbox{ for all } x-\bar x \in M \cap \mathbb{B}(0,\varepsilon_2).
\end{equation}

{\bf [Combining the estimation and simplification]} Now let $\varepsilon \le \min\{\widetilde{\varepsilon_0,}\varepsilon_2\}$ be such that $\max_{1 \le i \le m}[g_i(x)]_+ \le 1$ for all $x \in \mathbb{B}(\bar x, \varepsilon)$. Let $K$ be a compact set containing $\mathbb{B}(\bar x,\varepsilon) \cup \mathbb{B}(0,\varepsilon)$.  Denote the Lipschitz constant of $g_i$ over $K$ by $L_i$, i.e., $|g_i(x_1)-g_i(x_2)|\le L_i\|x_1-x_2\|$ for all $x_1,x_2 \in K$. Set $L:=\max_{1 \le i \le m}L_i$ and $\gamma:=\max\{\gamma_0 r^2|I|+\widetilde{\gamma_0},\gamma_2\}$.

 To arrive at  the conclusion, we only need to show that for any $x \in  \mathbb{B}(\bar x,\varepsilon)$,
 \[
   \di(x,S) \le c \bigg(\max_{1 \le i \le m}[g_i(x)]_+\bigg)^{\max\big\{\frac{2}{\kappa(n,2d)},\frac{1}{\beta(n-1)d^n}\big\}},
 \]
 where $c:=2\gamma+L\gamma^2.$ To see this, let us fix an arbitrary  $x \in  \mathbb{B}(\bar x,\varepsilon)$. Note that $\mathbb{R}^n=M+M^{\bot}$.  Then, one can decompose $x-\bar x=u+v$ for some $u \in M\cap \mathbb{B}(0,\varepsilon)$ and $v \in M^{\bot}\cap \mathbb{B}(0,\varepsilon)$.
 This together with (\ref{eq:001}) and (\ref{eq:0066}) implies that
\begin{align}
 \di(u+\bar x,S) &\le \gamma \Big(\max_{i \notin I}[g_i(u+\bar x)]_+\Big)^{\max\big\{\frac{2}{\kappa(k,2d)},\frac{1}{\beta(k-1)d^k}\big\}}\label{KolEeq:4}\\
 \|v\|& \le \gamma\bigg(\max_{i \in I}[g_i(v+\bar x)]_+\bigg)^{\max\big\{\frac{2}{\kappa(n-k,2d)},\frac{1}{\beta(n-k-1)d^{n-k}}\big\}}.\label{KolEeq:5}
\end{align}
Therefore,
\begin{align*}
 &\di(x,S)  \le  \di(u+\bar x,S)+\|x-(u+\bar x)\| \\
 &= \di(u+\bar x,S)+\|v \| \\
 & \le  \gamma \bigg(\max_{i \notin I}[g_i(u+\bar x)]_+\bigg)^{\max\big\{\frac{2}{\kappa(k,2d)},\frac{1}{\beta(k-1)d^k}\big\}} +\gamma  \bigg(\max_{i \in I}[g_i(v+\bar x)]_+\bigg)^{\max\big\{\frac{2}{\kappa(n-k,2d)},\frac{1}{\beta(n-k-1)d^{n-k}}\big\}} \\
 & = \gamma \bigg(\max_{i \notin I}[g_i(u+\bar x)]_+\bigg)^{\max\big\{\frac{2}{\kappa(k,2d)},\frac{1}{\beta(k-1)d^k}\big\}} +\gamma   \bigg(\max_{i \in I}[g_i(x)]_+\bigg)^{\max\big\{\frac{2}{\kappa(n,2d)},\frac{1}{\beta(n-1)d^{n}}\big\}} \\
 & \le  \gamma \bigg(\max_{1 \le i \le m}[g_i(u+\bar x)]_+\bigg)^{\max\big\{\frac{2}{\kappa(k,2d)},\frac{1}{\beta(k-1)d^k}\big\}} +\gamma  \bigg(\max_{1 \le i \le m}[g_i(x)]_+\bigg)^{\max\big\{\frac{2}{\kappa(n,2d)},\frac{1}{\beta(n-1)d^{n}}\big\}},
\end{align*}
where the second equality follows by (\ref{eq:directional_constant}) and $v+\bar x+u=x$.
 Note that
\begin{align}
&\big|\max_{1 \le i \le m }[g_i(u+\bar x)]_+-\max_{1 \le i \le m}[g_i(x)]_+\big|
 \le \max_{1 \le i \le m}|g_i(u+\bar x)-g_i(x)|\nonumber \\
 & \le  L \|u+\bar x-x\|
=  L \|v\| \nonumber\\
& \le  L \gamma \Big(\max_{i \in I}[g_i(v+\bar x)]_+\Big)^{\max\big\{\frac{2}{\kappa(n-k,2d)},\frac{1}{\beta(n-k-1)d^{n-k}}\big\}} \quad\text{(by \eqref{KolEeq:5})}\nonumber\\
& =  L \gamma \Big(\max_{i \in I}[g_i(x)]_+\Big)^{\max\big\{\frac{2}{\kappa(n-k,2d)},\frac{1}{\beta(n-k-1)d^{n-k}}\}}
\quad\text{(by \ref{eq:directional_constant}) and $v+\bar x+u=x$)}\label{KolEeq:7a}.
\end{align}
\allowdisplaybreaks
As $\max_{1 \le i \le m}[g_i(x)]_+ \le 1$ for all $x \in \mathbb{B}(\bar x, \varepsilon)$, it follows that
{\small \begin{align}\label{KolEeq:7}
\di(x,S) & \le  \gamma \bigg(\max_{1 \le i \le m}[g_i(u+\bar x)]_+\bigg)^{\frac{2}{\kappa(k,2d)}}
 \quad+\gamma \bigg(\max_{1 \le i \le m}[g_i(x)]_+\bigg)^{\max\big\{\frac{2}{\kappa(n,2d)},\frac{1}{\beta(n-1)d^{n}}\big\}}\nonumber \\
 & \le   \gamma \bigg( \max_{1 \le i \le m}[g_i(x)]_++L \gamma (\max_{i \in I}[g_i(x)]_+^{\frac{2}{\kappa(n-k,2d)}})\bigg)^{\frac{2}{\kappa(k,2d)}}\quad\text{(by \eqref{KolEeq:7a})} \nonumber\\ &\quad+\gamma \bigg(\max_{1 \le i \le m}[g_i(x)]_+\bigg)^{\max\big\{\frac{2}{\kappa(n,2d)},\frac{1}{\beta(n-1)d^{n}}\big\}} \nonumber\\
  & \le   \gamma \bigg( \max_{1 \le i \le m}[g_i(x)]^{\frac{2}{\kappa(n-k,2d)}} _++L \gamma(\max_{i \in I}[g_i(x)]_+^{\frac{2}{\kappa(n-k,2d)}})\bigg)^{\frac{2}{\kappa(k,2d)}}\nonumber\\
  &\quad +\gamma  \bigg(\max_{1 \le i \le m}[g_i(x)]_+\bigg)^{
 \max\big\{\frac{2}{\kappa(n,2d)},\frac{1}{\beta(n-1)d^{n}}\big\}} \nonumber\\
   & \leq       (\gamma+L\gamma^2)  \max_{1 \le i \le m}[g_i(x)]^{\frac{4}{\kappa(n-k,2d)\cdot\kappa(k,2d)}} _+
  +\gamma  \bigg(\max_{1 \le i \le m}[g_i(x)]_+\bigg)^{
   \max\big\{\frac{2}{\kappa(n,2d)},\frac{1}{\beta(n-1)d^{n}}\big\}}\nonumber \\
  &= (\gamma+L\gamma^2)  \max_{1 \le i \le m}[g_i(x)]^{\frac{4}{\big((2d-1)^{n-k}+1\big)\cdot\big((2d-1)^k+1\big)}} _+
+\gamma  \bigg(\max_{1 \le i \le m}[g_i(x)]_+\bigg)^{
  \max\big\{\frac{2}{\kappa(n,2d)},\frac{1}{\beta(n-1)d^{n}}\big\}} \nonumber\\
  &\leq   (\gamma+L\gamma^2)  \max_{1 \le i \le m}[g_i(x)]^{\frac{2}{(2d-1)^{n}+1}}_+
   +\gamma  \bigg(\max_{1 \le i \le m}[g_i(x)]_+\bigg)^{\max\big\{\frac{2}{\kappa(n,2d)},\frac{1}{\beta(n-1)d^{n}}\big\}}.
\end{align}}
Similarly, we also have
{\small \begin{align}
 \di(x,S) & \le  \gamma \bigg(\max_{1 \le i \le m}[g_i(u+\bar x)]_+\bigg)^{\frac{1}{\beta(k-1)d^k}}
+\gamma \bigg(\max_{1 \le i \le m}[g_i(x)]_+\bigg)^{\max\big\{\frac{2}{\kappa(n,2d)},\frac{1}{\beta(n-1)d^{n}}\}} \nonumber\\
 & \le   \gamma \bigg( \max_{1 \le i \le m}[g_i(x)]_++L \gamma (\max_{i \in I}[g_i(x)]_+^{\frac{1}{\beta(n-k-1)d^{n-k}}})\bigg)^{\frac{1}{\beta(k-1)d^k}}\quad\text{(by
 \eqref{KolEeq:7a})} \nonumber\\ &\quad+\gamma \bigg(\max_{1 \le i \le m}[g_i(x)]_+\bigg)^{\max\big\{\frac{2}{\kappa(n,2d)},\frac{1}{\beta(n-1)d^{n}}\big\}} \nonumber\\
  & \le   \gamma \bigg( \max_{1 \le i \le m}[g_i(x)]^{\frac{1}{\beta(n-k-1)d^{n-k}}} _++L \gamma(\max_{i \in I}[g_i(x)]_+^{\frac{1}{\beta(n-k-1)d^{n-k}}})\bigg)^{\frac{1}{\beta(k-1)d^k}}\nonumber\\
  &\quad +\gamma  \bigg(\max_{1 \le i \le m}[g_i(x)]_+\bigg)^{
 \max\big\{\frac{2}{\kappa(n,2d)},\frac{1}{\beta(n-1)d^{n}}\big\}} \nonumber\\
  &= (\gamma+L\gamma^2)  \max_{1 \le i \le m}[g_i(x)]^{\frac{1}{\big(\beta(n-k-1)d^{n-k}\big)\cdot\big(\beta(k-1)d^k\big)}} _+
 +\gamma  \bigg(\max_{1 \le i \le m}[g_i(x)]_+\bigg)^{
  \max\big\{\frac{2}{\kappa(n,2d)},\frac{1}{\beta(n-1)d^{n}}\big\}} \nonumber\\
  &\leq   (\gamma+L\gamma^2)  \max_{1 \le i \le m}[g_i(x)]^{\frac{1}{\beta(n-1)d^n}} _+
  +\gamma  \bigg(\max_{1 \le i \le m}[g_i(x)]_+\bigg)^{\max\big\{\frac{2}{\kappa(n,2d)},\frac{1}{\beta(n-1)d^{n}}\big\}},\nonumber
\end{align}} where the last inequality was obtained by the Chu-Vandermonde identity.

In combination  with \eqref{KolEeq:7} we obtain
{\small \begin{align*}
\di(x,S)& \le
(\gamma+L\gamma^2)  \max_{1 \le i \le m}[g_i(x)]^{\max\big\{ \frac{2}{(2d-1)^{n}+1} ,\frac{1}{\beta(n-1)d^n}\big\}} _+
   \quad+\gamma  \bigg(\max_{1 \le i \le m}[g_i(x)]_+\bigg)^{\max\big\{\frac{2}{\kappa(n,2d)},\frac{1}{\beta(n-1)d^{n}}\big\}}\\
   &=(2\gamma+L\gamma^2)  \max_{1 \le i \le m}[g_i(x)]^{\max\big\{ \frac{2}{(2d-1)^{n}+1} ,\frac{1}{\beta(n-1)d^n}\}} _+.
\end{align*}}
This completes the proof. \qed
\end{proof}

As a corollary, we obtain a local error bound result which is independent of the partition of the index set.
\begin{corollary}\label{prop:2.0}
Let $g_i$  be convex polynomials on $\mathbb{R}^n$ with degree at most $d$ for every $i=1,\cdots,m$. Let $S:=\{x \in \mathbb{R}^n\ \mid g_i(x) \le 0, i=1,\cdots,m\}$ and $\bar x \in S$. Then there exist $c, \varepsilon > 0$ such that
\[
\di(x,S)\le c \Big(\max_{1 \le i \le m} [g_i(x)]_+\Big)^{\tau} \quad  { whenever } \quad \|x-\overline{x}\|\le\varepsilon,
\]
where $[a]_+:=\max\{a,0\}$, $\tau:=\max\big\{\frac{2}{\kappa(n,2d)}, \frac{1}{\beta(n-1)d^n}\big\}=\frac{1}{\min\big\{\frac{(2d-1)^n+1}{2},\, \beta(n-1)d^n\big\}}$ and  $\beta(n-1)$ is the central binomial coefficient with respect to $n-1$.
\end{corollary}
\begin{proof}
Choose $\varepsilon$ small enough so that $\max_{1 \le i \le m} [g_i(x)]_+ \le 1$. Then,  the conclusion follows immediately from the preceding Theorem~\ref{Tprop:2.0} by noting that $[g_i(x)]_+ \le \big(\max_{1 \le i \le m} [g_i(x)]_+\big)^{\frac{2}{\kappa(n,2d)}}$ for each $i=1,\ldots,m$.  \qed
\end{proof}

\begin{remark}{\bf (Discussion of the  exponent)}\label{Remark:3.5} Let $g_i$  be convex polynomials on $\mathbb{R}^n$ with degree at most $d$ for every $i=1,\cdots,m$. Let $S:=\{x \in \mathbb{R}^n\ \mid g_i(x) \le 0, i=1,\cdots,m\}$ and $\bar x \in S$.   We now make some discussion on the exponent in our local error bound results.
\begin{itemize}
\item[{\rm (1)}] Theorem~\ref{Tprop:2.0} shows that in the case when $d=1$ or if there exists $x_0 \in \mathbb{R}^n$ such that $g_i(x_0)<0$, $i=1,\ldots,m$, we indeed obtain a Lipschitz type local error bound. That is to say, in these cases, there exist $c, \varepsilon > 0$ such that
\[
\di(x,S)\le c \, \max_{1 \le i \le m}[g_i(x)]_+  \quad  { whenever } \quad \|x-\overline{x}\|\le\varepsilon,
\]
where $[a]_+:=\max\{a,0\}$. To see this, if $d=1$, then $\frac{2}{\kappa(n,2d)}=1$ and $\frac{1}{\beta(n-1)d^n} \ge 1$. So, $\tau=\max\big\{\frac{2}{\kappa(n,2d)}, \frac{1}{\beta(n-1)d^n}\big\}=1$ and the conclusion follows immediately from Theorem~\ref{Tprop:2.0}. On the other hand, if there exists $x_0 \in \mathbb{R}^n$ such that $g_i(x_0)<0$, $i=1,\ldots,m$, then, $J_0=\emptyset$. So, the conclusion  follows immediately from the same theorem.
\item[{\rm (2)}] In the case when $n=1$, we see that $\frac{2}{\kappa(n,2d)}=\frac{1}{\beta(n-1)d^n}=\frac{1}{d}$. So, when each $g_i$ is a univariate convex polynomial, then there exist $c, \varepsilon > 0$ such that
\[
\di(x,S)\le c \big(\max_{1 \le i \le m} [g_i(x)]_+\big)^{\frac{1}{d}} \quad  { whenever } \quad \|x-\overline{x}\|\le\varepsilon.
\]
Note that, for the naive simple example $g_1(x):=x^d$, local error bound holds at $0$ with exponent $\frac{1}{d}$. This suggests that our result  matches what one might expect in the univariate case.

\item[{\rm (3)}]  On the other hand, in general, our estimation on the exponent will not be optimal.

For example, if the inequality system  consists of one single convex polynomial, Fact \ref{lemma:li0} shows that the exponent can be set as $\frac{1}{(d-1)^n+1}$ while our results produce a weaker exponent $\max\big\{\frac{2}{(2d-1)^n+1}, \frac{1}{\beta(n-1)d^n}\big\}$.
An interesting feature of the exponent $\frac{1}{(d-1)^n+1}$ in Fact \ref{lemma:li0} is that, in the  convex quadratic case, it collapses to $\frac{1}{2}$ which is independent of the
dimension of the underlying space and which agrees with the known result presented in \cite{Li_Wu}. By contrast, our estimate   $\max\big\{\frac{2}{3^n+1}, \frac{1}{\beta(n-1)2^n}\big\}$ depends heavily  on the dimension $n$.

Moreover, as indicated in Example \ref{ex:01}, the best possible exponent might be $\frac{1}{d^n}$ (see \cite{Kollar} for some relevant discussion regarding the best possible exponent for general nonconvex polynomial system).
It would be interesting to find how to could improve our estimate here. 
\end{itemize}
Making better sense of these estimates  will be one of our future research topics.\qede
\end{remark}

Given $D\subseteq\RR^n$, we set $\di^{r}(\cdot,D):=\big(\di(\cdot,D)\big)^{r}$ for every $r\in\RR$.

\begin{theorem}[H\"olderian regularity]\label{ThesumSet:1}
Let $\gamma_i \in \mathbb{N}$, $i=1,\ldots,m$, and  $g_{i,j}$ be are convex polynomials  on $\RR^n$ with degree $d \in \mathbb{N}$, $j=1,\ldots,\gamma_i$, $i=1,\ldots,m$.
Recall that $$C_i=\Big\{x\in\RR^n\mid g_{i,j}(x)\leq 0, j=1,\ldots,\gamma_i  \Big\} \mbox{ and } C=\bigcap_{i=1}^m C_i.$$
Let $\theta>0$ and $K\subseteq\RR^n$ be a compact set.
  Then  there exists  $c> 0$ such that
$$\di^{\theta} (x,C) \le c \bigg(\sum_{i=1}^m \di^{\theta}(x, C_i)\bigg)^{\tau}, \quad \forall x \in K,$$
where $\tau:=\frac{1}{\min\big\{\frac{(2d-1)^n+1}{2},\, \beta(n-1)d^n\big\}}$ and  $\beta(n-1)$ is the central binomial coefficient with respect to $n-1$ which is given by ${n-1 \choose {[(n-1)/2] }}$.
\end{theorem}

\begin{proof}
To see the conclusion, we only need to show that for each $\bar x \in \mathbb{R}^n$, there exist $c,\varepsilon>0$ such that
\begin{equation}\label{eq:sundi.yumen}
\di^{\theta} (x, C) \le c \bigg(\sum_{i=1}^m \di^{\theta}(x, C_i)\bigg)^\tau, \quad \textrm{for all } \|x-\bar x\| \le \varepsilon.
\end{equation}
Indeed, granting this and fixing a compact set $K$, then for any $\bar x \in K$ there exist $c_{\bar x},\varepsilon_{\bar x}>0$ such that
\[
\di^{\theta} (x, C) \le c_x\, \bigg(\sum_{i=1}^m  \di^{\theta}(x, C_i)\bigg)^\tau, \quad \textrm{for all } \|x-\bar x\| \le \varepsilon_x.
\]
As $K$ is compact and $\bigcup_{\bar x \in K}\mathbb{B}(\bar x;\varepsilon_{\bar x}) \supseteq K$, we can find finitely many points $\bar{x}_1,\cdots,\bar{x}_s \in K$, $s \in \mathbb{N}$, such that $\bigcup_{i=1}^s\mathbb{B}(\bar x_i;\varepsilon_{\bar x_i}) \supseteq K$. Let $c:=\max\{c_{\bar x_1},\cdots,c_{\bar x_s}\}$. Then, for any $x \in K$,
there exists $i_0 \in \{1,\cdots,s\}$ such that $x \in \mathbb{B}(\bar x_{i_0};\varepsilon_{\bar x_{i_0}})$, and hence
\[
 \di^{\theta}(x, C) \le c_{\bar x_{i_0}}\, \bigg(\sum_{i=1}^m  \di^{\theta}(x, C_i)\bigg)^\tau \le c \, \bigg(\sum_{i=1}^m \di^{\theta}(x, C_i)\bigg)^\tau.
\]

We now show \eqref{eq:sundi.yumen} holds. Fix
$\bar{x}\in\RR^n$. We consider two cases.

\emph{Case 1}: $\bar{x} \notin C$.

Then there exist $\varepsilon_1, \eta, M >0$ such that
\[
\sum_{i=1}^m  \di^{\theta}(x, C_i) \ge \eta  \mbox{ and }   \di^{\theta}(x,  C) \le M \quad \textrm{for all } \|x-\bar x\| \le \varepsilon_1.
\]
Therefore, $ \di^{\theta}(x, C) \le M = \frac{M}{\eta^{\tau}} \eta^{\tau} \le \frac{M}{\eta^{\tau}} \Big(\sum_{i=1}^m \di^{\theta}(x, C_i)\Big)^{\tau}$ for all $\|x-\bar x\| \le \varepsilon_1$ and hence, \eqref{eq:sundi.yumen} holds.

\emph{Case 2}: $\bar{x} \in C$.

We have
\begin{align*}
C=\Big\{x\in\RR^n\mid g_{i,1}(x)\leq0, g_{i,2}(x)\leq0,\cdots, g_{i,\gamma_i}(x)\leq 0,\quad i=1,\cdots,m\Big\}.
\end{align*}
By Corollary~\ref{prop:2.0},  there exist
positive constants $c_0$ and $\delta$ such that
\begin{align*}
\di(x, C)
& \leq c_0^{\frac{1}{\theta}} \bigg(\max_{1\leq i\leq m} \big\{[g_{i,1}(x)]_+,\cdots, [g_{i,\gamma_i}(x)]_+ \big\} \bigg)^{\tau},\quad \forall \|x - \bar{x} \| \le \delta.
\end{align*}
Hence
\begin{align}\label{eq:op}
\di^{\theta}(x, C)
& \leq c_0 \bigg(\max_{1\leq i\leq m} \big\{[g_{i,1}(x)]_+,\cdots, [g_{i,\gamma_i}(x)]_+ \big\} \bigg)^{\theta\tau},\quad \forall \|x - \bar{x} \| \le \delta.
\end{align}

Now we claim that there exists $\beta>0$ such that
\begin{align}\label{Sumclos:e1}
 \bigg(\max_{1\leq i\leq m} \big\{g_{i,1}(x)]_+,\cdots, [g_{i,\gamma_i}(x)]_+ \big\}\bigg)^{\theta} \leq \beta\sum_{i=1}^m \di^{\theta}(x, C_i),\quad \forall \|x - \bar{x} \| \le \delta.
\end{align}
Suppose to the contrary  that there exists  a sequence
$(x_k)_{k\in\NN}$ in $\mathbb{B}(\bar{x},\delta)$ such that
\begin{align}\label{Sumclos:e2}
\bigg(\max_{1\leq i\leq m} \big\{[g_{i,1}(x_{k})]_+,\cdots, [g_{i,\gamma_i}(x_{k})]_+ \big\}\bigg)^{\theta} > k\sum_{i=1}^m \di^{\theta}(x_k, C_i),\quad\forall k\in\NN.
\end{align}
Without loss of generality, we can assume that $g_{i,1},g_{i,2},\cdots g_{i,\gamma_i}$ have the Lipschitz constant $L>0$ on $\mathbb{B}(\bar{x},\delta)$ for every $i=1,2,\cdots,m$. Then, there exists  a subsequence
$(x_{k_l})_{l\in\NN}$ of $(x_k)_{k\in\NN}$, $1\leq i_0\leq m$
and $1\leq j_0\leq\gamma_{i_0}$ such that
\begin{align*}
\max_{1\leq i\leq m} \big\{[g_{i,1}(x_{k_l})]_+,\cdots, [g_{i,\gamma_i}(x_{k_l})]_+ \big\}=[g_{i_0,j_0}(x_{k_l})]_+,\quad\forall l\in\NN.
\end{align*}
It follows from \eqref{Sumclos:e2} that
\begin{align}\label{Sumclos:e3}
\big([g_{i_0,j_0}(x_{k_l})]_+\big)^{\theta}>k_l\sum_{i=1}^m \di^{\theta}(x_{k_l}, C_i),\quad\forall l\in\NN.
\end{align}
Then $[g_{i_0,j_0}(x_{k_l})]_+=g_{i_0,j_0}(x_{k_l})$ and hence for every $l\in\NN$,
\begin{align}\label{Sumclos:e4}
\big(g_{i_0,j_0}(x_{k_l})\big)^{\theta}>k_l\sum_{i=1}^m \di^{\theta}(x_{k_l}, C_i)=k_l\sum_{i=1}^m \big\|x_{k_l}-P_{i}(x_{k_l})\big\|^{\theta}\geq k_l\big\|x_{k_l}-P_{i_0}(x_{k_l})\big\|^{\theta}.
\end{align}
Since  $P_{i_0}(x_{k_l})\in C_{i_0}$, we have $g_{i_0,j_0}\big(P_{i_0}(x_{k_l})\big)\leq0$ and
$\|x_{k_l}-P_{i_0}(x_{k_l})\|\leq\|x_{k_l}-\bar{x}\|<\delta$
by $\bar{x}\in C_{i_0}$.
Combining this with \eqref{Sumclos:e4}, we have
\begin{align*}
L^{\theta}\big\|x_{k_l}-P_{i_0}(x_{k_l})\big\|^{\theta}\geq \Big(g_{i_0,j_0}(x_{k_l})-g_{i_0,j_0}\big(P_{i_0}(x_{k_l})\big)\Big)^{\theta}> k_l\big\|x_{k_l}-P_{i_0}(x_{k_l})\big\|^{\theta},\quad\forall l\in\NN.
\end{align*}
Hence we have $L^{\theta}>k_l$ for every $l\in\NN$, this contradicts the fact that $k_l\longrightarrow +\infty$.
Thus, \eqref{Sumclos:e1} holds.

Combining \eqref{Sumclos:e1} and \eqref{eq:op}, we see that
\begin{align*}
\di^{\theta}(x, C)
& \leq c_0\beta^{\tau} \Big(\sum_{i=1}^m \di^{\theta}(x, C_i)\Big)^{\tau},\quad \forall \|x - \bar{x} \| \le \delta,
\end{align*}
and so the conclusion follows. \qed
\end{proof}

\vspace{-0.1cm}
\setcounter{equation}{0}
\section{Convergence rate for the cyclic projection algorithm}\label{s:main}
\vspace{-0.1cm}
In this section, we derive explicit convergence rate  of the cyclic projection algorithm applied to finite intersections of basic semi-algebraic convex sets.

Before we come to our main result, we need the following useful lemma,
Lemma~\ref{recur}, which is a special case of  Alber and  Reich's result in  \cite{ARei1} . For the reader's convenience,  we provide a direct and self-contained proof.
\begin{lemma} [\bf Recurrence relationships]\label{recur} Let $p >0$, and let $\{\delta_k\}_{k=0}^{\infty}$ and $\{\beta_k\}_{k=0}^{\infty}$ be two sequences of nonnegative numbers satisfying the conditions
\[
\beta_{k+1} \le \beta_k(1-\delta_k \beta_k^{p})\;\mbox{ as }\;k=0,1,\ldots.
\]
Then, we have
\begin{equation} \label{eq:sundi.pangpang}
\beta_{k}\le\bigg(\beta_0^{-p}+\displaystyle p \sum_{i=0}^{k-1} \delta_i \bigg)^{-\frac{1}{p}}\;\mbox{ for all }\;k\in\mathbb N.
\end{equation}We use the convention that $\frac{1}{0}=+\infty$.
In particular, we have $\displaystyle\lim_{k\rightarrow\infty}\beta_k=0$ whenever $\displaystyle\sum_{k=0}^{\infty}\delta_k=\infty$.
\end{lemma}
\begin{proof} It follows from our assumption that
$$
0\le \beta_{i+1}\le\beta_{i}\le\cdots\le\beta_0\;\mbox{ and }\;\delta_i \beta_{i}^{p+1}\le\beta_{i}-\beta_{i+1}\;\mbox{ as }\;i\in\mathbb N.
$$

Fix $k\in\NN$.  We consider two cases.

\emph{Case 1}: $\beta_k=0$.

Clearly, \eqref{eq:sundi.pangpang} holds.

\emph{Case 1}: $\beta_k\neq 0$.

Thus $\beta_k>0$ and hence $\beta_i>0$ for every $i\leq k$.
Define the nonincreasing function $h:\RR_{++}\rightarrow\RX$ by   $h(x):=x^{-(p+1)}$. As $\delta_i \, h(\beta_{i})^{-1}=\delta_i \beta_{i}^{p+1} \le \beta_i-\beta_{i+1}$,
then we get
\[
\delta_i \le
(\beta_{i}-\beta_{i+1})h(\beta_{i})\le
\int_{\beta_{i+1}}^{\beta_i}h(x)dx=
\frac{\beta_{i+1}^{-p}-\beta_i^{-p}}{p}.
\]
This implies that
\begin{equation}\label{eq:sundi.haha}
\beta_{i+1}^{-p}-\beta_i^{-p}\ge p \delta_i \;\mbox{ for all }\;i\in\mathbb{N}\cup\{0\}.
\end{equation}

Now fix any $k\in \mathbb{N}$ and, summing (\ref{eq:sundi.haha}) from $i=0$ to $i=k-1$, we get
\[
\beta_{k}^{-p}-\beta_0^{-p}\ge p \, \sum_{i=0}^{k-1} \delta_i . 
\]
which implies the conclusion in (\ref{eq:sundi.pangpang}). \qed
\end{proof}

We also need the following technical result. The proof of it follows in part that of \cite[Lemmas~3\&4]{GPR}, one may also consult \cite{Heinz}.

\begin{proposition}[Cyclic convergence rate]\label{TheMTR:1}
Let $D_i\subseteq\RR^n$ be a closed convex set, $\forall i=1,\ldots,m$, and  $ D:=\bigcap_{i=1}^m D_i\neq\varnothing$. Let $x_0\in\RR^n$ and the sequence of cyclic projections, $(x_k)_{k\in\NN}$, be defined by
\begin{align*}
x_1:=P_1 x_0,\, x_2:=P_2 x_1,\,\cdots,\, x_m:=P_mx_{m-1},\, x_{m+1}:=P_1x_{m}\ldots,
\end{align*} where we set $P_i:=P_{D_i}$ for the convenience.
 Suppose that H\"{o}lderian regularity with exponent $\tau$ $(0<\tau \le 1)$  holds: for any compact set $K \subseteq \mathbb{R}^n$ and $\theta>0$, there exists $c_0>0$ such that
$$\di^{\theta} (x,D) \le c_0 \bigg(\sum_{i=1}^m \di^{\theta}(x, D_i)\bigg)^{\tau}, \quad \forall x \in K.$$
Then $x_k$ converges to  $x_{\infty}\in D$. Moreover, there exist $M>0$ and $r_0 \in \left]0,1\right[$ such that
\begin{eqnarray*}
\|x_{k}-x_{\infty}\| \leq
\left\{\begin{array}{ccc}
M \, k^{-\dfrac{1}{2\tau^{-1}-2}}& \mbox{ if } & \tau \in \left]0,1\right[, \\
M \, r_0^k & \mbox{ if } & \tau=1,
\end{array} \ \quad\forall k\in\NN,
\right.
\end{eqnarray*}

\end{proposition}

\begin{proof}
We denoted by  $\alpha_i:=(i\mod m)+1,\forall i\in\NN$.
Thus $x_{k+1}=P_{D_{\alpha_k}}x_k$. By Fact~\ref{BregFms:1},
there exists $x_{\infty}\in D$ such that $x_{k}\longrightarrow x_{\infty}$.

We first follow closely the proofs of \cite[Lemmas~3\&4]{GPR} to get that
\begin{align}
\di^2(x_{k}, D)-\di^2(x_{k+1}, D)&\geq\di^2(x_{k}, D_{\alpha_k}),\quad\forall
k\in\NN.
\label{TheMTR:E1}
\end{align}
Indeed, using the definition of projection operator, we have
\begin{align*}
&\di^2(x_{k}, D)-\di^2(x_{k+1}, D)\geq\|x_{k}-P_D x_k\|^2-\|x_{k+1}-P_D x_{k}\|^2\\
&=\|x_{k}-P_D x_k\|^2-\|P_{\alpha_k}x_{k}-P_D x_{k}\|^2=\|x_{k}-P_D x_k\|^2-\big\|P_{\alpha_k}x_{k}-x_k+x_k-P_D x_{k}\big\|^2\\
&=\|x_{k}-P_D x_k\|^2-\|P_{\alpha_k}x_{k}-x_k\|^2-\|x_{k}-P_D x_k\|^2
+2\big\langle x_k-P_{\alpha_k}x_{k},x_k-P_D x_{k}\big\rangle\\
&\geq-\|x_k-P_{\alpha_k} x_{k}\|^2
+2\big\langle x_k-P_{\alpha_k}x_{k},x_k-P_{\alpha_k}x_k+P_{\alpha_k}x_k-P_D x_{k}\big\rangle\\
&=-\|x_k-P_{\alpha_k} x_{k}\|^2+2\|x_k-P_{\alpha_k}x_k\|^2
+2\big\langle x_k-P_{\alpha_k}x_{k},P_{\alpha_k}x_k-P_D x_{k}\big\rangle\\
&\geq\|x_k-P_{\alpha_k}x_k\|^2=\di^2(x_k, D_{\alpha_k}),\quad\forall k\in\NN.
\end{align*}
Hence \eqref{TheMTR:E1} holds.

Next we claim that for every $i\in\NN$
\begin{align}
\di(x_{k}, D_{\alpha_i})\leq\di(x_{k}, D_{\alpha_k})+\di(x_{k+1}, D_{\alpha_{k+1}})+\ldots+\di(x_{k+m-1}, D_{\alpha_{k+m-1}}).\label{TheMTR:Ea1}
\end{align}
To see this, note that there exists $i_0\leq m-1$ such that $\alpha_{i_0+k}=\alpha_{i}$. Then, we have
\begin{align*}
&\di(x_{k}, D_{\alpha_i})=\di(x_{k}, D_{\alpha_{i_0+k}})
\leq\|x_k-x_{i_0+k}\|+\big\|x_{i_0+k}-P_{\alpha_{i_0+k}}x_{i_0+k}\big\|\\
&\leq\|x_k-x_{k+1}\|+\cdots+\|x_{i_0+k-1}-x_{i_0+k}\|+\big\|x_{i_0+k}-P_{\alpha_{i_0+k}}x_{i_0+k}\big\|\\
&=\|x_k-P_{\alpha_k}x_{k}\|+\cdots+\|x_{i_0+k-1}-P_{\alpha_{i_0+k-1}}x_{i_0+k-1}\|
+\big\|x_{i_0+k}-P_{\alpha_{i_0+k}}x_{i_0+k}\big\|\\
&=\di(x_k, D_{\alpha_k})+\di(x_{k+1}, D_{\alpha_{k+1}})+\cdots+\di(x_{i_0+k}, D_{\alpha_{i_0+k}})\\
&\leq\di(x_k, D_{\alpha_k})+\di(x_{k+1}, D_{\alpha_{k+1}})+\cdots+\di(x_{k+m-1}, D_{\alpha_{k+m-1}})\quad\text{(by $i_0\leq m-1$)}.
\end{align*}
Hence \eqref{TheMTR:Ea1} holds.

Thus by \eqref{TheMTR:Ea1},
\begin{align}
&\di^2(x_{k}, D_{\alpha_i})\leq \Big(m \max_{k\leq i\leq k+m-1}\di(x_{i}, D_{\alpha_i})\Big)^2\nonumber\\
&\leq m^2\Big(\di^2(x_{k}, D_{\alpha_k})+\di^2(x_{k+1}, D_{\alpha_{k+1}})+\ldots+\di^2(x_{k+m-1}, D_{\alpha_{k+m-1}})\Big).\label{TheMTR:Ec1}
\end{align}

By the assumption, there exists $c_0>0$ such that
\begin{align*}
\di^2(x_k, D) \le c_0 \Big(\sum_{i=1}^m\di^2(x_k,D_i)\Big)^{\tau},\quad\forall k\in\NN.
\end{align*}
By enlarging $c_0$ if necessary, we may assume that $c_0>1$.

Let $r:=\tau^{-1}$. Then by \eqref{TheMTR:Ec1}, for every $k\in\NN$,
\begin{align}
&\frac{1}{c_0^{r}}\di^{2r}(x_k, D)\nonumber\\
&\leq  \sum_{i=1}^m\di^2(x_k,D_i)\leq m \max_{1\leq i\leq m} \di^2(x_k,D_i)\nonumber\\
&\leq m^3\Big(\di^2(x_{k}, D_{\alpha_k})+\di^2(x_{k+1}, D_{\alpha_{k+1}})+\ldots+\di^2(x_{k+m-1}, D_{\alpha_{k+m-1}})\Big)\nonumber\\
&\leq m^3\sum_{i=k}^{k+m-1}\di^2(x_{i}, D)-\di^2(x_{i+1}, D)
\quad\text{(by \eqref{TheMTR:E1})}\nonumber\\
&=m^3\Big(\di^2(x_{k}, D)-\di^2(x_{k+m}, D)\Big).\label{TheMTR:E2}
\end{align}
Thus we have
\begin{align}
&\di^2(x_{k+m}, D)\leq
\di^2(x_{k}, D)-\frac{1}{m^3c_0^{r}}\di^{2r}(x_k, D).\label{TheMTR:E3}
\end{align}
Now we consider two cases.

\emph{Case 1}: $\tau \in \left]0,1\right[$.

Thus $r>1$.
 Fix $k_0\in\NN$.
Let $\beta_i:=\di^2(x_{k_0+im}, D)$, $\forall i \in \NN\cup\{0\}$.
Then \eqref{TheMTR:E3} shows that
\begin{align}
\beta_{i+1} \leq \beta_i-\frac{1}{m^3c_0^{r}}\beta_i^{r}
 =\beta_i(1-\frac{1}{m^3c_0^{r}}\beta_i^{r-1}).
\end{align}
By Lemma~\ref{recur},
\begin{align*}
\di^2(x_{k_0+im}, D)&= \beta_{i}\le\bigg(\beta_0^{1-r}+(r-1)i\frac{1}{m^3c_0^r} \bigg)^{-\frac{1}{r-1}},\quad \forall i\in\NN.
\end{align*}
Thus, there exists $M_0>0$ such that
\begin{align*}
\di(x_{k_0+im}, D)\leq M_0\frac{1}{\sqrt[2(r-1)]{i}},\quad \forall i\in\NN.
\end{align*}
Hence
there exists $M_1>0$ such that
\begin{align}
\di(x_{k}, D)\leq M_1\frac{1}{\sqrt[2(r-1)]{k}}=M_1 k^{-\dfrac{1}{2\tau^{-1}-2}},\quad\forall k\in\NN.\label{TheMTR:E4a}
\end{align}
So, we have
\begin{align}
\|x_k-x_{\infty}\|\leq\|x_k-P_{D}(x_k)\|+\|P_{D}(x_k)-x_{\infty}\|=
\di(x_{k}, D)+\|P_{D}(x_k)-x_{\infty}\|.\label{TheMTR:E5}
\end{align}
By \cite[Lemma~3]{GPR},
\begin{align*}
\|x_{k+l}-P_{D}(x_k)\|\leq \|x_{k}-P_{D}(x_k)\|=\di(x_{k}, D),\quad\forall l\in\NN.
\end{align*}
Letting $l\longrightarrow \infty$ in the above inequality, we obtain that
\begin{align}
\|x_{\infty}-P_{D}(x_k)\|\leq \di(x_{k}, D).\label{TheMTR:E6}
\end{align}
Combining \eqref{TheMTR:E5}, \eqref{TheMTR:E6} and \eqref{TheMTR:E4a},
\begin{align}
\|x_k-x_{\infty}\|\leq 2\di(x_{k}, D)\leq 2M_1k^{-\dfrac{1}{2\tau^{-1}-2}},\quad\forall k\in\NN.
\end{align}
Thus, the conclusion follows by letting $M:=2M_1$.

\emph{Case 2}: $\tau=1$.

Then we have $r=\tau^{-1}=1$, and so, (\ref{TheMTR:E3}) implies that for all $k \in \NN$
\[
\di(x_{k+m}, D)\leq
\sqrt{1-\frac{1}{m^3c_0^{r}}}\, \di(x_k, D).
\]
Hence
there exist $M_1'>0$ and $r_0 \in \left]0,1\right[$ such that
$\di(x_{k}, D)\leq M_1' r_0^k,\quad\forall k\in\NN.$
Then, using a similar method of proof as in Case 1, we obtain that
\begin{align}
\|x_k-x_{\infty}\|\leq 2\di(x_{k}, D)\leq 2M_1' r_0^k ,\quad\forall k\in\NN.
\end{align}
Thus, the conclusion follows by letting $M:=2M_1'$.
\qed
\end{proof}
\begin{remark}{\bf (Connection to the existing result on linear convergence)}
In the case where there exists $i_0 \in \{1,\ldots,m\}$ such that  $D_{i_0} \cap {\rm int}\, \big(\bigcap_{i \neq i_0}D_i\big) \neq \emptyset$ (in this case, we say the intersection is regular), then the H\"{o}lderian regularity result holds with exponent
$\tau=1$. So, the preceding proposition implies that the cyclic projection algorithm converges linearly in the regular intersection case.
Thus, this recovers the linear convergence result for cyclic projection algorithm established in \cite{GPR}.\qede
  \end{remark}

We are now ready for one of our main results.

\begin{theorem}[Estimate of the cyclic convergence rate]\label{TheMTR:3}
  Let $x_0\in\RR^n$ and the sequence of cyclic projections, $(x_k)_{k\in\NN}$, be defined by
\begin{align*}
x_1:=P_1 x_0,\, x_2:=P_2 x_1,\,\cdots,\, x_m:=P_mx_{m-1},\, x_{m+1}:=P_1x_{m}\ldots.
\end{align*}
Then $x_k$ converges to  $x_{\infty}\in C$, and there exist $M>0$ and $r_0\in\left]0,1\right[$ such that
\begin{align*}
\|x_{k}-x_{\infty}\|\leq \begin{cases}M\frac{1}{k^{\rho}},\,&\text{if}\,\, d>1;\\
M r^k_0,\,&\text{if}\,\, d=1
\end{cases},\quad\forall k\in\NN,
\end{align*}
where $\rho:=\frac{1}{\min\big\{(2d-1)^n-1,\, 2\beta(n-1)d^n-2\big\}}$ and $\beta(n-1)$ is the central binomial coefficient with respect to $n-1$ which is given by ${n-1 \choose {[(n-1)/2] }}$.
\end{theorem}

\begin{proof}
Combining Theorem \ref{ThesumSet:1} and Proposition~\ref{TheMTR:1}, we directly obtain $\tau:=\frac{1}{\min\big\{\frac{(2d-1)^n+1}{2},\, \beta(n-1)d^n\big\}}$. Note that $2\tau^{-1}-2=\min\big\{(2d-1)^n-1,\, 2\beta(n-1)d^n-2\big\}$. Thus the conclusion follows from the preceding proposition. \qed
\end{proof}

\begin{remark}{\bf (Discussion on our estimation of the convergence rate)}
Although our estimate of the convergence rate works for cyclic projection algorithm with finitely many basic semialgebraic convex sets without any regularity condition, the estimated convergence rate is quite poor when the dimension $n$ of the underlying space and the maximal degree $d$ are large. This is mainly due to the fact that the estimated convergence rate is derived by using the local error bound result for general convex polynomial systems. It would be interesting to see how one could improve the estimation of the convergence rate by either adopting other approaches or by further exploiting the structure of the underlying convex sets. For example, one possibility would be to examine problems involving some  suitable additional curvature or uniform
convexity assumptions. This will be another of our future research topics.
\end{remark}

\subsection{Alternating projection algorithm}
In this subsection, we discuss the convergence rate of the alternating projection algorithm. We assume throughout this subsection that
 \begin{equation*}
   \addtolength{\fboxsep}{5pt}
   \boxed{
   \begin{alignedat}{4}
   &g_i,h_j\,\text{
are convex polynomials with degree at most $d$},\,  \forall i=1,2,\cdots,m,\, j=1,2,\cdots,l\\
     & A:=\{x  \in \mathbb{R}^n\mid g_i(x) \le 0, i=1,\cdots,m\}\\
&B:=\{x \in \mathbb{R}^n\mid h_j(x) \le 0, j=1,\cdots,l\}\\
&b_0\in\RR^n,\quad a_{k+1}:=P_A b_k, \quad b_{k+1}:=P_Ba_{k+1}.
\end{alignedat}
   }
\end{equation*}

As an immediate corollary of Theorem \ref{TheMTR:3}, we first obtain the following estimate on the convergence rate of the alternating projection algorithm in the case where the two sets have nonempty intersection.
\vspace{-0.1cm}
\subsection*{The case of two sets with nonempty intersection}\label{subsecf}
\begin{corollary}[Alternating convergence rate]\label{ATMCoR:1}
Suppose that $A \cap B \neq \emptyset$. Let the sequence $\{(a_k,b_k)\}$ be generated by the alternating projection algorithm. Then, $a_k,b_k \longrightarrow c \in A \cap B$. Moreover, there exist  $M>0$ and $r_0\in\left]0,1\right[$ such that for every $k\in\NN$
\begin{align*}
\|a_k-c\|\leq  \begin{cases}M\frac{1}{k^{\rho}},\,&\text{if}\,\, d>1;\\
M r^k_0,\,&\text{if}\,\, d=1
\end{cases}, \quad&\mbox{ and } \quad\|b_k-c\|\leq  \begin{cases}M\frac{1}{k^{\rho}},\,&\text{if}\,\, d>1;\\
M r^k_0,\,&\text{if}\,\, d=1
\end{cases},
\end{align*}
where $\rho:=\frac{1}{\min\big\{(2d-1)^n-1,\, 2\beta(n-1)d^n-2\big\}}$ and $\beta(n-1)$ is the central binomial coefficient with respect to $n-1$ which is given by ${n-1 \choose {[(n-1)/2] }}$.
\end{corollary}

Recently, \cite{Attouch1} established a local convergence rate analysis for proximal alternating projection methods for
very general nonconvex problems, and where the corresponding convergence rate involves the exponent of the Kurdyka-$\L$ojasiewicz inequality. The \emph{proximal alternating projection method} is a variant of the alternating projection algorithm we discussed here. On the other hand, as we discussed before, in general, the actual exponent of the Kurdyka-$\L$ojasiewicz inequality is typically unknown and hard to estimate. Corollary \ref{ATMCoR:1} above complements the result of \cite{Attouch1} in the case of basic convex semialgebraic cases by providing an explicit estimate of the convergence rate.

\vspace{-0.1cm}
\subsection*{The case of two sets with empty intersection}\label{subsecf}
In this part, we  consider the general case where the intersection of these two sets
 is (possibly) empty. We first need the following lemma.
 
\begin{lemma}\label{LemmClos:1}
 The difference  $B-A$ of  two basic semi-algebraic sets $A,B$ is closed.
\end{lemma}

\begin{proof}
Let $b_k \in B$ and $a_k \in A$ be such that $b_k-a_k \longrightarrow c$. We now show that $c \in B-A$. Consider  the following convex polynomial optimization problem
\begin{eqnarray*}
(P) & \min_{x \in \mathbb{R}^n, y \in \mathbb{R}^n} & \|(y-x)-c\|^2 \\
& \mbox{ s.t. } & g_i(x) \le 0, i=1,\cdots,m, \\
& & h_j(y) \le 0, j=1,\cdots,l.
\end{eqnarray*}
Note that $(a_k,b_k)$ are feasible for (P). Hence we see that $\inf(P)=0$. By Fact~\ref{lemma:2.4},  the optimal
solution of (P) exists. Thus there exists $x \in A$ and $y \in B$ such that $c=y-x \in B -A$. Hence the conclusion follows. \qed
\end{proof}

\begin{remark} \label{cor:attain} With A and B defined as above,
Fact~\ref{lemma:2.4} implies that $B-A$ is closed convex. Hence $P_{B-A}0\neq\varnothing$.
Let $v:=P_{B-A}0$. Then, there exist $a \in A$ and $b \in B$ such that
$v=b-a$ and hence  $\di(A,B)=\|v\|$. \qede
\end{remark}

\begin{remark}
In  general,  the distance between two convex and semi-algebraic sets need not be attained.
For instance, consider  $D:=\{(x_1,x_2)\in\RR^2\mid x_1x_2 \ge 1, x_1 \ge 0, x_2 \ge 0\}$ and $E:=\{(x_1,x_2) \in \mathbb{R}^2\mid x_1=0\}$. It is clear that $D,E$ are both convex and semialgebraic; while $D$ is not a basic semi-algebraic convex set (as explained in Example~\ref{remark:1}).
Clearly, $\di(D,E)=0$ but $D \cap E=\emptyset$. Thus, the distance is not attained in this case. \qede
\end{remark}

The proof of Theorem~\ref{TheAPM:1} partially follows that of \cite[Theorem~3.12]{Heinz_Set}.

\begin{theorem}{\bf (Convergence rate in the infeasible case)}\label{TheAPM:1} Let the sequence $\{(a_k,b_k)\}$ be generated by the alternating projection algorithm.
 Then $a_k \longrightarrow \widetilde{a }\in A$ and $b_k \longrightarrow \widetilde{b} \in B$ with $\widetilde{b}-\widetilde{a }=v$
 where $v:=P_{B-A}0$. Moreover, there exist $M>0$ and $r_0\in\left]0,1\right[$ such that for every $k\in\NN$
\begin{align}
\|a_k-\widetilde{a}\|\leq
 \begin{cases}M\frac{1}{k^{\rho}},\,&\text{if}\,\, d>1;\\
M r^k_0,\,&\text{if}\,\, d=1
\end{cases}\quad \text{ and }\quad \|b_k,-b\| \le \begin{cases}M\frac{1}{k^{\rho}},\,&\text{if}\,\, d>1;\\
M r_0^k,\,&\text{if}\,\, d=1
\end{cases},
\end{align}
where $\rho:=\frac{1}{\min\big\{(2d-1)^n-1,\, 2\beta(n-1)d^n-2\big\}}$ and $\beta(n-1)$ is the central binomial coefficient with respect to $n-1$ which is given by ${n-1 \choose {[(n-1)/2] }}$.
\end{theorem}

\begin{proof}
Lemma~\ref{LemmClos:1} implies that $B-A$ is closed. Then
by Fact~\ref{FactPr:1}(i), there exist $\widetilde{a} \in A,
\widetilde{b} \in B$ such that
$a_k \longrightarrow \widetilde{a }\in A$ and $b_k \longrightarrow \widetilde{b} \in B$ with $\widetilde{b}-\widetilde{a }=v$.
By Theorem~\ref{ThesumSet:1}, there exists $c_0>1$ such that
\begin{align}\label{ProRE:1}
\di(a_k, A \cap (B-v)) \le c_0 \big(\di(a_k,A)+\di(a_k,B-v)\big)^{\frac{1}{r}}= c_0 \di^{\frac{1}{r}}(a_k,B-v),
\end{align}
where $r:=\min\big\{\frac{(2d-1)^n+1}{2}, \beta(n-1)d^n\big\}$.
Fix $x \in A \cap (B-v)$. Note that $v=P_{B-A}0$ we have $P_B(x)=x+v$ by Fact~\ref{FactPr:1}(ii). This implies that
\begin{align*}
\di^2(a_k,B-v) &\le \|a_k-(b_k-v)\|^2= \|(a_k-x)-\big(b_k-(v+x)\big)\|^2 \\
&=  \|(a_k-x)-(P_Ba_k-P_Bx)\|^2 \\
& \le  \|a_k-x\|^2-\|P_Ba_k-P_Bx\|^2\quad \text{(by \cite[Proposition~4.8]{BC2011})}\\
& =  \|a_k-x\|^2-\|b_k-(x+v)\|^2 \\
& \le  \|a_k-x\|^2-\|P_Ab_k-P_A(x+v)\|^2 \\
& =  \|a_k-x\|^2-\|a_{k+1}-x\|^2\quad\text{(by Fact~\ref{FactPr:1}(ii))}.
\end{align*}
In particular, choose $x=P_{A \cap (B-v)}a_k$. Then, we have
\begin{eqnarray*}
\di^2(a_k,B-v) & \le & \di^2(a_k,A \cap (B-v))-\|a_{k+1}-P_{A \cap (B-v)}a_k\|^2 \\
&\le & \di^2(a_k,A \cap (B-v))-\di^2(a_{k+1},A \cap (B-v)).
\end{eqnarray*}
Combining with \eqref{ProRE:1}, we have
\begin{align}
\frac{1}{c_0^{2r}}\di^{2r}(a_k, A \cap (B-v))& \le \di^2(a_k,(B-v))\nonumber\\
 &\le \di^2(a_k,A \cap (B-v))-\di^2(a_{k+1},A \cap (B-v)).\label{PTsSrt:E1}
\end{align}
Thus
\[
\di(a_{k+1},A \cap (B-v))^2 \le \di(a_k,A \cap (B-v))^2 - \frac{1}{c_0^{2r}}\di(a_k, A \cap (B-v))^{2r}.
\]
Now, let $\beta_k:=\di(a_k,A \cap (B-v))^2$, $k \in \mathbb{N}$. Then, we have
\begin{align}
\beta_{k+1} \le \beta_k(1-\frac{1}{c_0^{2r}}\beta_k^{r-1}).\label{LreAl:1}
\end{align}
Now we consider two cases.

\emph{Case 1}: $d>1$.

In this case, we have $r>1$.
Applying the preceding Lemma~\ref{recur} with $\delta_k:=\frac{1}{c_0^{2r}}$ and $p:=r-1$, by \eqref{LreAl:1},
we see that
\[
\di^2(a_k,A \cap (B-v))= \beta_{k}\le\bigg(\beta_0^{1-r}+\displaystyle \frac{(r-1)}{c_0^{2r}} k \bigg)^{-\frac{1}{r-1}}\;\mbox{ for all }\;k\in\mathbb N.
\]
Thence there exists $M_0>0$ such that
\[
\di(a_k,A \cap (B-v)) \le M_0 \frac{1}{k^{\rho}},\quad\forall k\in\NN
\]
where $\rho:=\frac{1}{\min\big\{(2d-1)^n-1,\, 2\beta(n-1)d^n-2\big\}}$. Then,
\cite[Example~3.2]{Heinz_Set} shows that $(a_k)_{k\in\NN}$ is Fej\'{e}r monotone with respect to
$A \cap (B-v)$.
Thus, by Fact~\ref{FactPr:2},
\begin{align*}
\|a_k-\widetilde{a}\|\leq2\di(a_k,A \cap (B-v)) \leq 2M_0 \frac{1}{k^{\rho}},\quad\forall k\in\NN.
\end{align*}

\emph{Case 2}: $d=1$.

Thus $r=1$. Then by \eqref{LreAl:1}, $\di(a_{k+1},A \cap (B-v)) \le \theta\di(a_k, A \cap (B-v))$,
where $\theta:=\sqrt{1-\frac{1}{c_0^{2}}}$. Then $\theta <1$ since $c_0>1$.
  Hence there exists $M_1>0$ such that
\begin{align*}\di(a_{k},A \cap (B-v)) \le M_1 \theta^k.
\end{align*}
From \cite[Example~3.2]{Heinz_Set}, we see that $(a_k)_{k\in\NN}$ is Fej\'{e}r monotone with respect to
$A \cap (B-v)$.
Thus, by Fact~\ref{FactPr:2},
\begin{align*}
\|a_k-\widetilde{a}\|\leq 2\di(a_{k},A \cap (B-v)) \leq 2M_1 \theta^k.
\end{align*}
Set $M_2:=\max\{2M_0,2M_1\}$. Combining the above two cases, we have
\begin{align*}
\|a_k-\widetilde{a}\|\leq
 \begin{cases}M_2\frac{1}{k^{\rho}},\,&\text{if}\,\, d>1;\\
M_2 \theta^k,\,&\text{if}\,\, d=1
\end{cases},\quad\forall k\in\NN.
\end{align*}

Similarly,  we can show that there exist $L>0$ and $\eta\in\left]0,1\right[$ such that
\[
\|b_k-\widetilde{b}\|\leq \begin{cases}L\frac{1}{k^{\rho}},\,&\text{if}\,\, d>1;\\
L \eta^k,\,&\text{if}\,\, d=1
\end{cases},\quad\forall k\in\NN.
\]
Therefore, the conclusion follows by taking $M:=\max\{M_2,L\}$ and $r_0:=\max\{\theta,\eta\}$. \qed
\end{proof}

\setcounter{equation}{0}
\section{Examples and remarks}{\label{s:final}

In this section, we will provide several examples of the rates of  convergence of the cyclic projection algorithm
and  the von Neumann alternating projection algorithm. We first start with some examples where the basic semialgebraic convex
sets are described by convex quadratic functions. Subsequently, we will examine examples where the basic semialgebraic convex
sets are described by higher degree convex polynomials.

\vspace{-0.1cm}
\subsection*{Basic semialgebraic convex sets described by convex quadratic functions}
\begin{example}
Let \begin{align*}C_1&:=\{(x,y)\in\RR^2\mid (x+1)^2+y^2-1\leq0\}\\
C_2&:=\{(x,y)\in\RR^2\mid x+y-1\leq0\}\\
C_3&:=\{(x,y)\in\RR^2\mid (x-1)^2+ y^2-1\leq0\}\\
C_4&:=\{(x,y)\in\RR^2\mid x+(y+2)^2-4\leq0\}.
\end{align*}
Take $x_0\in\RR^2$. Let $(x_k)_{k\in\NN}$ be defined by
\begin{align*}
x_1:=P_1 x_0, \,x_2:=P_2 x_1,\, x_3:=P_3 x_2,\, x_4:=P_4x_{3},\, x_{5}:=P_1x_{4}\ldots
\end{align*}
Then
$\|x_k\|= O(\frac{1}{k^{\frac{1}{6}}})$.
\qede \end{example}
\begin{proof}
Clearly, $\bigcap_{i=1}^4C_i=\{0\}$.  Then apply  $n=2$ and $ d=2$ to Theorem~\ref{TheMTR:1}.
\end{proof} \qed

\begin{example}\label{Examp:1Gefor}
Let $\alpha\geq0$ and  \begin{align*}A&:=\big\{(x,y)\in\RR^2\mid (x+1)^2+y^2-1\leq0\big\}=(-1,0)+\overline{\mathbb{B}}(0,1)\\
B&:=\big\{(x,y)\in\RR^2\mid -x+\alpha\leq0\big\}.
\end{align*}
 Let $(a_k)_{k\in\NN}:=(u_k, v_k)_{k\in\NN}$ and $(b_k)_{k\in\NN}
 :=(s_k, t_k)_{k\in\NN}$ be defined by
\begin{align*}
b_0\in\RR^2,\quad a_{k+1}:=P_A b_k, \quad b_{k+1}:=P_Ba_{k+1}.
\end{align*}
Then for every $k\geq 2$
\begin{align*}
b_{k}&=\Big(\alpha, \frac{t_1}{\sqrt{(1+\alpha)^{2(k-1)}+t^2_1
\sum_{i=0}^{k-2}((1+\alpha)^{2i}}}\Big)\\
a_{k+1}&=\Big(-1+\frac{\alpha+1}{\sqrt{(\alpha+1)^2+\frac{t^2_1}{{(1+\alpha)^{2(k-1)}+t^2_1
\sum_{i=0}^{k-2}((1+\alpha)^{2i}}}}},
\frac{t_1}{\sqrt{(1+\alpha)^{2k}+t^2_1
\sum_{i=0}^{k-1}((1+\alpha)^{2i}}} \Big).
\end{align*}
Consequently, $a_k\longrightarrow0$ and $b_k\longrightarrow (\alpha,0)$  at the rate of $k^{-\frac{1}{2}}$ when $\alpha=0$.
When $\alpha\neq0$(then $A\cap B=\varnothing$), $a_k\longrightarrow0$ and $b_k\longrightarrow (\alpha,0)$  at the rate of $(1+\alpha)^{-k}$. \qede
\end{example}

\begin{proof}
We first claim that
\begin{align}
b_{k+1}=\big(\alpha, t_{{k+1}}\big)=\big(\alpha, \frac{t_k}{\sqrt{(1+\alpha)^2+t^2_k}}\big),\quad\forall k\geq 1.\label{ExAlPr:a1}
\end{align}
By \cite[Examples~3.17\&3.21 and Proposition~3.17]{BC2011}, we have
\begin{align}\label{ExAlPr:1}
P_A(x,y)&=(-1,0)+\frac{(x+1,y)}{\max\big\{1,\|(x+1,y)\|\big\}},\quad\forall (x,y)\in\RR^2\\
P_B(x,y)&=(\alpha,y),\quad\forall (x,y)\notin \inte B.\nonumber
\end{align}
Let $k\geq 1$. Since $A\cap B=\{0\}$ or $A\cap B=\varnothing$, $a_{k}\notin \inte B$. Then by \eqref{ExAlPr:1},  $b_k=(\alpha, v_{k})$ and then
\begin{align*}
a_{k+1}&=P_Ab_k=(-1,0)+\frac{(1+\alpha,v_{k})}{\max\big\{1,\|(1+\alpha,v_{k})\|\big\}}
=(-1,0)+\frac{(1+\alpha,v_{k})}{\sqrt{(1+\alpha)^2+v^2_{k}}}\\
b_{k+1}&=P_B(a_{k+1})=(\alpha, \frac{v_{k}}{\sqrt{(1+\alpha)^2+v^2_{k}}})=(\alpha, \frac{t_{k}}{\sqrt{(1+\alpha)^2+t^2_{k}}}).
\end{align*}
Hence \eqref{ExAlPr:a1} holds.
Next we show that
\begin{align}
b_{k}=\Big(\alpha, \frac{t_1}{\sqrt{(1+\alpha)^{2(k-1)}+t^2_1
\sum_{i=0}^{k-2}((1+\alpha)^{2i}}}\Big),\quad\forall k\geq 2.\label{ExAlPr:3}
\end{align}
We prove \eqref{ExAlPr:3} by the induction on $k$.

By \eqref{ExAlPr:a1}, \eqref{ExAlPr:3} holds when $k=2$.  Now assume  that
 \eqref{ExAlPr:3} holds when $k=p$, where $p\geq 2$. Now we consider the case of $k=p+1$.  By the assumption, we have
 \begin{align}
b_{p}=\Big(\alpha, \frac{t_1}{\sqrt{(1+\alpha)^{2(p-1)}
+t^2_1\sum_{i=0}^{p-2}((1+\alpha)^{2i}}}\Big).\label{ExAlPr:4}
\end{align}
Then by \eqref{ExAlPr:a1}, we have
\begin{align*}
b_{p+1}&=\Big(\alpha, \frac{t_{p}}{\sqrt{(1+\alpha)^2+t^2_{p}}}\Big)\\
&=\Big(\alpha, \dfrac{ \dfrac{t_1}{\sqrt{(1+\alpha)^{2(p-1)}
+t^2_1\sum_{i=0}^{p-2}((1+\alpha)^{2i}}}}
{\sqrt{(1+\alpha)^2+\dfrac{t^2_1}
{(1+\alpha)^{2(p-1)}+t^2_1\sum_{i=0}^{p-2}(1+\alpha)^{2i}}}}
\Big)\\
&=\Big(\alpha, \frac{t_1}{\sqrt{(1+\alpha)^{2p}+t^2_1
\sum_{i=0}^{p-1}((1+\alpha)^{2i}}}\Big).
\end{align*}
Hence \eqref{ExAlPr:3} holds.

Combining\eqref{ExAlPr:1} and \eqref{ExAlPr:3}, we have for every $k\geq 2$
\begin{align*}
&a_{k+1}=P_A b_k\\
&=\Big(-1+\frac{\alpha+1}{\sqrt{(\alpha+1)^2
+\frac{t^2_1}{{(1+\alpha)^{2(k-1)}+t^2_1
\sum_{i=0}^{k-2}((1+\alpha)^{2i}}}}},
\frac{t_1}{\sqrt{(1+\alpha)^{2k}+t^2_1
\sum_{i=0}^{k-1}((1+\alpha)^{2i}}} \Big).
\end{align*}
Hence $a_k\longrightarrow0$ and $b_k\longrightarrow (\alpha,0)$  at the rate of $k^{-\frac{1}{2}}$ when $\alpha=0$.
When $\alpha\neq0$, $a_k\longrightarrow0$ and $b_k\longrightarrow (\alpha,0)$  at the rate of $(1+\alpha)^{-k}$.
\end{proof} \qed

\begin{remark}
According to Theorem~\ref{TheAPM:1},we can only deduce that $(a_k)_{k\in\NN}$   in Example~\ref{Examp:1Gefor} converges to
$(0,0)$ and $(b_k)_{k\in\NN}$ converge to
$(\alpha,0)$ at the rate of at least of $k^{-\frac{1}{6}}$. \qede
\end{remark}

\begin{example}\label{ex:5.5}
Let \begin{align*}A&:=\big\{(x,y)\in\RR^2\mid (x+1)^2+y^2-1\leq0\big\}\\
B&:=\big\{(x,y)\in\RR^2\mid (x-1)^2+y^2-1\leq0\big\}.
\end{align*}
 Let $(x_k, y_k)_{k\in\NN}$ be defined by
\begin{align*}
(x_0,y_0)\in\RR^2, (x_1, y_1):=P_A (x_0,y_0), (x_2, y_2):=P_B (x_1,y_1), (x_3, y_3):=P_A (x_2,y_2),\quad \cdots.
\end{align*}

Note that   \begin{align*}P_A (x,y)&=\left(-1+\frac{x-1}{\sqrt{(x+1)^2+y^2}},\frac{y}{\sqrt{(x+1)^2+y^2}}\right),\quad  \forall (x,y) \in \RR_+\times\RR_{++}\\
P_B
(x,y)&=\left(1+\frac{x+1}{\sqrt{(x-1)^2+y^2}},\frac{y}{\sqrt{(x-1)^2+y^2}}\right),\quad
\forall  (x,y) \in \RR_-\times\RR_{++}.
\end{align*}
Figure \ref{fig:2circ}  depicts the algorithm's trajectory with
starting point $(0,2)$.
\begin{figure}\begin{center}
\includegraphics[scale=0.63,angle=90]{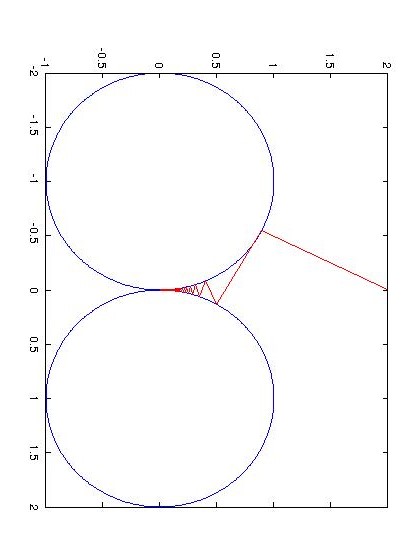}
\caption{The iteration commencing at $(0,2)$.}\label{fig:2circ}
\end{center}
\end{figure}

Suppose, without loss of generality, that one starts on a point on one the half-circles nearest the other circle.
Then the distance from zero (for every $k\in\NN$), $r_k:=\sqrt{x^2_k+y^2_k}$ satisfies
$r_k^2=2 \alpha_k$
where $\alpha_k:=|x_k|$ since $(x_k,y_k)\in\bd A\cup\bd B$. Hence
\[1-\alpha_{k+1}= \frac{1+\alpha_k}{\sqrt{1+4\alpha_k}}.\]
Linearizing, we obtain that $w_k:=4\alpha_k$ approximately satisfies the logistics equation
\[w_{k+1} \approx w_k(1-w_k)\]
This can be explicitly solved by writing
\[\frac 1{w_{k+1}}-\frac 1{w_{k}}=\frac{1}{1-w_k}\]
When summing and dividing by $N$, leads to
\begin{align}\lim_{N\to \infty}\frac 1{N w_N} &=\lim_{N\to \infty}\left(\frac 1{N w_N}-\frac 1{N w_{0}}\right)\\\nonumber &=\lim_{N\to \infty}\frac{1}N\sum_{k=0}^{N-1}\frac{1}{1-w_k}\\& = \lim_{N\to \infty}\frac1{1-w_N}=1,\nonumber
\end{align}
since C\'esaro summability is conservative and $w_N \to 0$.
Hence $\alpha_k \thicksim 1/(4k)$ and so \[\sqrt{x_k^2+y_k^2} = r_k \thicksim \frac 1{\sqrt{2k}}.\]
For instance, with $\alpha_0=1,N=10^6$, we obtain $\alpha_N \approx 0.0000002499992442.$
A similar analysis can be performed in the previous example.
\qede \end{example}

\begin{remark}
According to Theorem~\ref{TheAPM:1}, we can only deduce that $(a_k)_{k\in\NN}$ and  $(b_k)_{k\in\NN}$  in Example~\ref{ex:5.5} converge to
$(0,0)$ at the rate of at least of $k^{-\frac{1}{6}}$. \qede
\end{remark}

\subsection*{Basic semialgebraic convex sets described by convex polynomials}
In general, identifying the exact convergence rate of the cyclic
projection method in a direct way can be quite arduous when applied to finitely many basic semialgebraic convex sets. Below, we provide two simple examples to
illustrate how the convergence rate depends on the maximum degree of the polynomials that described the basic semialgebraic convex sets and on the dimension of the underlying space.

\begin{example}\label{Examp:1La}
Let $A,B$ be defined by \begin{align*}
A&:=\big\{(x,y)\in\RR^2\mid x\leq0\big\}\\
B&:=\big\{(x,y)\in\RR^2\mid y^d-x\leq0\big\},
\end{align*}
where $d$ is an even number.
 Let $(a_k)_{k\in\NN},\, (b_k)_{k\in\NN}
 :=(x_k, y_k)_{k\in\NN}$ be defined by
\begin{align*}
b_0:=(x_0,y_0)\in\RR\times \RR_{++}\,\text{with}\, \|b_0\|\leq 1,\quad a_{k+1}:=P_A b_k, \quad b_{k+1}:=P_Ba_{k+1}.
\end{align*}
Then for every $k\in\NN$
\begin{align*}
a_{k+1}=(0,y_k),\quad b_k=(y^d_k, y_k),\quad\text{and}\quad
d \, y_{k+1}^{2d-1}+y_{k+1}=y_k.
\end{align*}
Consequently,
$(a_k)_{k\in\NN}$ and $(b_k)_{k\in\NN}$ converge to
$0$ at the rate of at least of $k^{-\frac{1}{2d-2}}$.
\end{example}
\begin{proof}
Let $k\in\NN$. Since $A\cap B=\{0\}$, $b_{k}\notin \inte A$.
Since $b_k\in\bd B$ and then $x_k=y^d_k$. Thus
$a_{k+1}=P_Ab_{k}=P_A(x_{k},y_{k})=(0,y_{k})$.  Then we have $y_k \longrightarrow 0$,
\begin{align*}b_{k+1}=(x_{k+1},y_{k+1})=(y^d_{k+1},y_{k+1})=P_B a_{k+1}=P_B(0,y_{k}).
\end{align*}
Thus $y_{k+1}$ is a minimizer of the function
\begin{align*}
y\mapsto\|(y^d,y)-(0,y_{k})\|^2=
\|(y^d,y-y_{k})\|^2=y^{2d}+y^2-2yy_k+y^2_k.
\end{align*}
Thus
$(d \, y_{k+1}^{2d-2}+1) y_{k+1}-y_k=d \, y_{k+1}^{2d-1}+y_{k+1}-y_k=0.$
Then we have
\[
\frac{1}{(d \, y_{k+1}^{2d-2}+1)^{2d-2}} \frac{1}{y_{k+1}^{2d-2}}=\frac{1}{y_k^{2d-2}}
\]
Now, for the function $h(x):=\frac{1}{(x+1)^{2d-2}}$, we have $$h'(x)=-(2d-2)\frac{1}{(x+1)^{2d-1}} \mbox{ for all } x \in \mathbb{R}.$$
Note that $y_k \longrightarrow 0$. So, for all large $k$,
\[
 \frac{1}{(d \, y_{k+1}^{2d-2}+1)^{2d-2}} =h(d \, y_{k+1}^{2d-2}) \approx h(0) + h'(0) d \, y_{k+1}^{2d-2}= 1-d(2d-2)y_{k+1}^{2d-2}.
\]
This gives us that for all large $k$,
\[
\frac{1}{y_k^{2d-2}}\approx \big(1-d(2d-2)y_{k+1}^{2d-2}\big) \frac{1}{y_{k+1}^{2d-2}}=  \frac{1}{y_{k+1}^{2d-2}} -d(2d-2).
\]
In other words, there exists $l_0 \in \mathbb{N}$ such that for all $l \ge l_0$
\begin{equation}\label{eq:yyu}
  \frac{1}{y_{l+1}^{2d-2}}-\frac{1}{y_l^{2d-2}} \approx d(2d-2).
\end{equation}
Let $k>l_0$ . Summing the above relation from $l=l_0$ to $k$, we have
\[
 \frac{1}{y_{k}^{2d-2}}-\frac{1}{y_{l_0}^{2d-2}} \approx d(2d-2)(k-l_0+1)
\]
and so,
\[
y_{k} \approx  \bigg(\frac{1}{y_{l_0}^{2d-2}} + d(2d-2)(k-l_0+1)\bigg)^{-\frac{1}{2d-2}}= O(k^{-\frac{1}{2d-2}}).
\]
%
%
%
Thence $(a_k)_{k\in\NN}$ and $(b_k)_{k\in\NN}$ converge to
$0$ at that rate.
\end{proof} \qed

\begin{remark}
Similarly, according to Theorem~\ref{TheAPM:1}, we can only deduce that $(a_k)_{k\in\NN}$ and  $(b_k)_{k\in\NN}$  in Example~\ref{Examp:1La} converge to
$(0,0)$ at the rate of at least of ${k^{-\frac{1}{2d^2-2}}}$.
\end{remark}


\begin{example}
Let \begin{align*}A&:=\big\{(x_1,\ldots,x_n)\in\RR^n\mid (x_1+1)^4+\sum_{i=2}^nx_i^4-1\leq0\big\}\\
B&:=\big\{(x_1,\ldots,x_n)\in\RR^n\mid (x_1-2)^4+\sum_{i=2}^nx_i^4-1\leq0\big\}.
\end{align*}
 Let $(a_k)_{k\in\NN}, (b_k)_{k\in\NN}$ be defined by
\begin{align*}
b_0\in\RR^2,\quad a_{k+1}:=P_A b_k, \quad b_{k+1}:=P_Ba_{k+1}.
\end{align*}
Then
$\|a_k\|= O(\frac{1}{k^{\rho_n}}), \|b_k-(1,0,\ldots,0))\|= O(\frac{1}{k^{\rho_n}})$ with $\rho_n=\frac{1}{\min\big\{7^n-1,\, 2\beta(n-1)4^n-2\big\}}$.
\qede \end{example}
\begin{proof} By the assumption, there exist unique points $a_0\in\bd A,\, b_0\in\bd B$ such that $1=\di(A,B)=\|a_0-b_0\|$. Clearly, $a_0=(0,0,\ldots,0)$ and $b_0=(1,0,\ldots,0)$.
  Then, the conclusion follows by applying Theorem~\ref{TheAPM:1} with $d=4$.
\end{proof} \qed

\vspace{-0.1cm}
\section{Conclusion and Open Questions}
In this paper, we studied the rate of convergence of the cyclic projection algorithm applied to  finitely many
basic semi-algebraic convex sets.  We established an explicit convergence rate estimate
 which relies on the maximum degree of the polynomials that generate the basic semi-algebraic convex sets and the dimension of the underlying space. We also examined some concrete  examples and compared the actual convergence rate with our estimate.

 Although our estimate of the convergence rate works for cyclic projection algorithm with finitely many basic semialgebraic convex sets without any regularity condition, the limitation of our approach is that the estimated convergence rate behaves quite badly when the dimension $n$ of the underlying space and the maximal degree $d$ are large. Our results have suggested the following future research topics and open questions:
 \begin{itemize}
   \item The explicit examples (Examples \ref{Examp:1Gefor}, \ref{ex:5.5} and \ref{Examp:1La}) show that, in general, our estimate of the convergence rate of the cyclic projection algorithm will not be tight. It would be interesting to see how one can sharpen the estimate obtained in this paper and get a tight estimate for the cyclic projection algorithm. In particular, finding the right exponent when each set is defined by convex quadratic functions would be a good starting point.
       \item Can we extend the approach here to analyze the convergence rate of the Douglas-Rachford algorithm? Almost nothing is known except for affine sets.
 \end{itemize}
 These will be our future research topics and will be examined later on.
 \vspace{-0.1cm}
\section*{Acknowledgments}
\vspace{-0.1cm}
The authors are grateful to
Dr.\ Simeon Reich and  the two anonymous referees and the editor
for their pertinent and constructive comments.
Jonathan  Borwein, Guoyin Li and Liangjin Yao were partially
supported by various Australian Research Council grants.


\footnotesize



 \begin{thebibliography}{1}
 \bibitem{ARei1}
 Y.\ Alber and S.\ Reich,
An iterative method for solving a class of nonlinear operator equations in Banach spaces,
\emph{Panamerican Mathematical Journal}, vol.~ 4, pp.~39--54, 1994.



\bibitem{Attouch1} H. Attouch, J. Bolte, P. Redont and A. Soubeyran, Proximal alternating minimization and projection methods for
nonconvex problems: an approach based on the
Kurdyka-$\L$ojasiewicz inequality, \emph{Mathematics of Operation Research},
Vol. ~35, no. 2,pp.~438-457, 2010.

\bibitem{Attouch2} H. Attouch, J. Bolte and B. Svaiter,  Convergence of descent methods for semi-algebraic and tame problems: proximal algorithms, forward-backward splitting, and regularized Gauss-Seidel methods. \emph{Mathematical Programming} Vol. 137 , no. 1-2, Ser. A, pp. 91-129, 2013.



\bibitem{Auslender}  A. Auslender and  M. Teboulle,  \emph{Asymptotic cones and functions in optimization and variational inequalities}. Springer Monographs in Mathematics. Springer-Verlag, New York, 2003.

\bibitem{Auslender2000} A. Auslender, Existence of optimal solutions and duality results under weak conditions, \emph{Mathematical Programming}
Vol. ~ 88,  pp~45-59, 2000.




\bibitem{BaComt}
J.-B.\ Baillon, P. L.\ Combettes, and R. Cominetti,
There is no variational characterization of the cycles in the method of periodic projections,
\emph{Journal of Functional Analysis}, vol.~262, pp.~400--408, 2012.

\bibitem{Bausc2}
H.H.\ Bauschke,
Projection algorithms: results and open problems, \emph{Inherently parallel algorithms in feasibility and optimization and their applications}, pp.~11–-22,
 2001.

\bibitem{Heinz_Set} H.H. Bauschke and J.M. Borwein, On the convergence of von Neumann’s alternating projection algorithm
for two sets, \emph{Set-Valued Analysis}, vol.~1, pp.~185--212, 1993.

\bibitem{BBJAT1} H.H. Bauschke and J.M. Borwein,
Dykstra's alternating projection algorithm for two sets,
\emph{Journal of Approximation Theory }, vol.~79, pp.~418--443, 1994.

\bibitem{Heinz} H.H. Bauschke and J.M. Borwein: On projection algorithms for solving convex feasibility problems, \emph{SIAM Review}, vol.~38, pp. 367-426, 1996.


\bibitem{Bausc1}
H.H. Bauschke, J.M. Borwein, and A.S. Lewis, The method of cyclic projections for closed convex sets in Hilbert space, \emph{Recent developments in optimization theory and nonlinear analysis}, Contemporary Mathematic, pp.~ 1--38, 1997


\bibitem{BC2011}
H.H.\ Bauschke and P.L.\ Combettes,
\emph{Convex Analysis and Monotone Operator Theory in Hilbert Spaces},
Springer, 2011.

\bibitem{BLPW}
H.H.\ Bauschke, D.R.\ Luke, H. M.\ Phan, and X.\ Wang,
Restricted normal cones and the method of alternating projections,
\emph{Set-Valued and Variational Analysis}, in press, 
\url{http://arxiv.org/abs/1205.0318v1}.

\bibitem{Belousov1} E.G.\ Belousov, \emph{Introduction to convex analysis and integer programming}, Izdat. Moskov. Univ., Moscow, 1977


\bibitem{Belousov}  E.G.\  Belousov, On types of Hausdorff
discontinuity from above for convex closed mappings, \emph{Optimization},
vol.~49, pp.~303--325, 2001.









\bibitem{convex_polynomial} E. G. Belousov and D. Klatte, A Frank-Wolfe type theorem for convex
polynomial programs, Computational Optimization and Applications,
vol.~22, pp. 37-48, 2002.




\bibitem{real} J. Bochnak,  M. Coste and M.F. Roy,
\emph{Real algebraic geometry},  Springer-Verlag, Berlin, 1998.

\bibitem{BorVan}
J.M.\ Borwein and J.D.\ Vanderwerff,
\emph{Convex Functions},
Cambridge University Press, 2010.

\bibitem{Bregman}
L.M.\ Br\`{e}gman,
Finding the common point of convex sets by the method of successive projection. \emph{Doklady Akademii Nauk SSSR }, vol.~162,  pp.~487--490, 1965.



 \bibitem{BrBSi1}
 R.E Bruck and S.\ Reich,
 Nonexpansive projections and resolvents of accretive operators in Banach spaces,
 \emph{Houston Journal of Mathematics}, vol.~3, pp.~459--470, 1977.








\bibitem{Deus}
F.\ Deutsch,
\emph{Best approximation in inner product spaces},
CMS Books in Mathematics/Ouvrages de Math\'{e}matiques de la SMC, 7, Springer-Verlag, New York, 2001.

%


 \bibitem{effective}  J.  Gwo\'{z}dziewicz, The {\L}ojasiewicz exponent of an analytic function at an isolated zero, Commentarii Mathematici Helvetici, vo.~74, pp.~364--375, 1999.


 \bibitem{GPR}L.G.\ Gubin; B.T.\ Polyak, and E.V. Raik,
The method of projections for finding the common point of convex sets,
\emph{USSR Computational Mathematics and Mathematical Physics},   vol.~7,  pp.~1--24, 1967.

\bibitem{Nie}
{\sc J.~W.~Helton, and J.~W.~Nie,}  Semidefinite representation of convex sets, \emph{Mathematical Programming}, Vol 122, Ser. A, pp.~21--64, 2010.

 \bibitem{Klatte1} D. Klatte, Hoffman's error bound for systems of convex inequalities.  Mathematical programming with data perturbations,  185--199, Lecture Notes in Pure and Appl. Math.,  Dekker, New York, 1998.
\bibitem{Kollar}
J.\ Koll\'{a}r,
An effective {\L}ojasiewicz inequality for real polynomials,
\emph{Periodica Mathematica Hungarica}, vol.~38, pp.~213--221, 1999.

\bibitem{KoRei}
E.\ Kopeck\'{a} and  S.\ Reich,
A note on the von Neumann alternating projections algorithm,
\emph{Journal of Nonlinear and Convex Analysis}, vol~ 5, pp.~379--386, 2004.

\bibitem{Pang} A.S.\ Lewis and J. S. Pang, Error bounds for convex inequality systems
Generalized Convexity, Generalized Monotonicity  J.P. Crouzeix, J.E. Martinez-Legaz and M. Volle (eds), pp.~75--110, 1998.

\bibitem{Li0} G.\ Li, On the asymptotic well behaved functions and global error bound for convex polynomials,  \emph{SIAM Journal on Optimization}, vol.~20, pp.~1923--1943, 2010.



\bibitem{Li1}  G.\ Li and K.F.\ Ng,  Error bounds of generalized D-gap functions for nonsmooth and nonmonotone variational inequality problems, \emph{SIAM Journal on Optimization},  vol.~20, pp.~667--690, 2009.

\bibitem{Li_Boris} G.\ Li and B.S.\ Mordukhovich, H\"{o}lder metric subregularity with applications to proximal point method, \emph{SIAM Journal on Optimization}, vol.~22,  pp.~1655--1684, 2012.



 \bibitem{Li_Wu} W. Li, Error bounds for piecewise convex quadratic programs and
applications, \emph{SIAM J. on Control and Optim.}, vol.~33, pp.~1510--
1529, 1995.


\bibitem{Luo_Luo} X.D.\ Luo and Z.Q.\ Luo, Extension of Hoffman's
error bound to polynomial systems, \emph{SIAM Journal on Optimization}, vol.~4,pp.~383--392, 1994.


\bibitem{Luo}  Z.Q.\ Luo, J.S.\  Pang, and D.\  Ralph, Mathematical Programs with Equilibrium Constraints. Cambridge University Press, Cambridge, 1996.


\bibitem{quasi_convex} W.T.\ Obuchowska, On generalizations of the Frank-Wolfe theorem to convex and
quasi-convex programmes, Computational Optimization and
Applications, vol.~33, pp.~349--364, 2006.

\bibitem{Pang_MP} J.S.\ Pang, Error bounds in mathematical programming. \emph{Mathematical Programming}, vol.~79,
pp.~299-332, 1997.


\bibitem{ph}
R.R.\ Phelps,
\emph{Convex Functions, Monotone Operators and
Differentiability},
2nd Edition, Springer-Verlag, 1993.


\bibitem{Robinson} S.M. Robinson, \emph{Regularity and Stability for Convex Multivalued Functions}, \emph{ Mathematics of Operations Research}, vol.~1,  pp.~130--143, 1976.

\bibitem{RockWets}
R.T.\ Rockafellar and R.J-B Wets,
\emph{Variational Analysis}, 3nd Printing,
Springer-Verlag, 2009.

\bibitem{Shironin}
V.M\ Shironin, \emph{On Hausdorff continuity of convex and convex polynomial mappings}, in Mathematical
Optimization: Questions of Solvability and Stability, E.G.\ Belousov and B.\ Bank,
eds., Moscow University Publishing, Moscow, 1986.

\bibitem{vonnum:1}
J.\ von Neumann, \emph{Functional Operators, Vol.II},
Princeton University Press, 1950.


\bibitem{Zalinescu}
C.\ Z\u{a}linescu,
\emph{Convex Analysis in General Vector Spaces}, World Scientific
Publishing, 2002.

\end{thebibliography}
\end{document}